\documentclass[12pt]{article}
\usepackage{graphicx}
\usepackage{color,latexsym,amsfonts,amssymb}
\usepackage{bm}     
\usepackage{amsthm} 
\usepackage{amsmath}
\usepackage{booktabs} 
\usepackage{boxedminipage}  
\usepackage{natbib} 
\usepackage{algorithm,algpseudocode}  
\usepackage{slashbox}   
\usepackage{subfigure} 

\DeclareMathOperator*{\argmax}{argmax}
\DeclareMathOperator*{\argmin}{argmin}

\newtheorem{theorem}{Theorem}

\newtheorem{lemma}{Lemma}
\newtheorem{definition}{Definition}

\newcommand{\CVaR}{\mathrm{CVaR}}
\newcommand{\VaR}{\mathrm{VaR}}

\title{\vspace{-1.5cm} Risk-Sensitive Markov Decision Processes with Long-Run CVaR Criterion}
\author{Li Xia\thanks{L. Xia is with the School of Business, Sun Yat-Sen University, Guangzhou 510275, China. (email: xiali5@sysu.edu.cn)}, \quad Peter W. Glynn\thanks{P. W. Glynn is with the Department of Management Science and Engineering, Stanford University, CA 94305, USA. (email: glynn@stanford.edu) }}
\date{}

\begin{document}
\maketitle

\vspace{-0.2cm}
\begin{abstract}
CVaR (Conditional Value at Risk) is a risk metric widely used in
finance. However, dynamically optimizing CVaR is difficult since it
is not a standard Markov decision process (MDP) and the principle of
dynamic programming fails. In this paper, we study the
infinite-horizon discrete-time MDP with a long-run CVaR criterion,
from the view of sensitivity-based optimization. By introducing a
pseudo CVaR metric, we derive a CVaR difference formula which
quantifies the difference of long-run CVaR under any two policies.
The optimality of deterministic policies is derived. We obtain a
so-called Bellman local optimality equation for CVaR, which is a
necessary and sufficient condition for local optimal policies and
only necessary for global optimal policies. A CVaR derivative
formula is also derived for providing more sensitivity information.
Then we develop a policy iteration type algorithm to efficiently
optimize CVaR, which is shown to converge to local optima in the
mixed policy space. We further discuss some extensions including the
mean-CVaR optimization and the maximization of CVaR. Finally, we
conduct numerical experiments relating to portfolio management to
demonstrate the main results. Our work may shed light on dynamically
optimizing CVaR from a sensitivity viewpoint.
\end{abstract}


\section{Introduction}\label{section_intro}
The Markov decision process (MDP) is a fundamental mathematical
model used to handle stochastic dynamic optimization problems
\citep{Feinberg02,Puterman94}. The study of MDPs in operations
research is multidisciplinary. It is deeply connected with
reinforcement learning in computer science
\citep{Kaelbling96,Sutton18}, optimal control in control science
\citep{Bertsekas05,Lewis12}, and dynamic discrete choice modeling in
econometrics \citep{Aguirregabiria10,Rust87}, etc. Traditional MDP
theory focuses on the criteria of discounted or long-run average
cost, where the principle of dynamic programming plays a key role.
However, for other optimization criteria, such as risk metrics in
finance, the corresponding optimization problems usually do not fit
the standard MDP model and specific investigations are needed case
by case.

Conditional value at risk (CVaR), also called average VaR or
expected shortfall, is a widely used risk metric, built upon the VaR
and variance related risk metrics \citep{Rockafellar00}. We view a
continuous random variable $X$ as the stochastic loss and denote its
distribution function as $F_X(x)$. The CVaR corresponding to this
stochastic loss at probability level $\alpha$ is defined as
$\CVaR(X) := \mathbb E[X|X \geq F_X^{-1}(\alpha)]=
\frac{1}{1-\alpha}\int_{\alpha}^{1} F_X^{-1}(\gamma)d\gamma$, which
measures the conditional expectation of losses greater than a given
quantile $F_X^{-1}(\alpha)$ (or called $\VaR(X)$), where
$0<\alpha<1$. Compared with other risk metrics such as variance or
VaR, CVaR can measure not only the down-side risk but also the
expected value of large losses. Moreover, CVaR is a risk coherent
measure \citep{Artzner99}, which has desirable properties
(monotonicity, translation equivariance,
sub-additivity, and positive homogeneity). 

CVaR has been widely used to measure risk in static (single-stage)
optimization problems \citep{Alexander06,Fu09,Hong14}. However, it
is difficult to handle the CVaR optimization problem in stochastic
dynamic (multi-stage) scenarios because of time inconsistency
\citep{Boda06,Pflug16}. In the MDP terminology, we may define an
instantaneous cost function for the long-run CVaR metric as $f(i,a)
:= F_X^{-1}(\alpha)+\frac{1}{1-\alpha}[c(i,a) -
F_X^{-1}(\alpha)]^+$, where $c(i,a)$ is the MDP's one-step loss
incurred at state $i$ with action $a$, $[\cdot]^+ :=
\max\{0,\cdot\}$, and $X$ is the random variable indicating the
stochastic loss whose value is realized as $c(i,a)$'s. We can
observe that $f(i,a)$ involves $F_X^{-1}(\alpha)$ which is affected
by future losses and actions. Therefore, the CVaR cost function
depends on future behaviors and is not additive or Markovian. The
CVaR dynamic optimization problem does not fit a standard MDP model.
Thus, the classical Bellman optimality equation does not hold and
the principle of dynamic programming fails. Fortunately,
\cite{Rockafellar02} discovered that the $\CVaR$ of the random
variable $X$ is equivalent to $\min\limits_{y\in \mathbb
R}\{y+\frac{1}{1-\alpha}\mathbb E[X - y]^+\}$, where
$y^*=F_X^{-1}(\alpha)$ exactly achieves the minimum. With this
equivalence, \cite{Bauerle11} skillfully converted the CVaR
minimization of discounted accumulated costs with finite or infinite
horizon into a bilevel optimization problem, where the outer one is
a static optimization problem with the auxiliary variable $y \in
\mathbb R$ and the inner one is a standard MDP with the fixed $y$
and an augmented state space. Although the existence and properties
of optimal policies were studied there, the efficient solution of
such bilevel MDP problems was not presented. For different values of
$y \in \mathbb R$, we have to solve different inner MDP problems,
which is computationally exhaustive. Such techniques of equivalent
MDP transformation and state augmentation were further extended to
study the CVaR minimization in other general cases, such as
semi-MDPs \citep{Huang16}, continuous-time MDPs \citep{Miller17},
unbounded costs \citep{Ugurlu17}, just to name a few. On the other
hand, \cite{Haskell15} used occupation measures to study the CVaR
optimization of discounted cost infinite-horizon MDPs from the
viewpoint of mathematical programming, where the optimality of
deterministic policies was also discussed. However, it is usually
not efficient to solve a series of mathematical programs, especially
when they are not convex or linear programs. Efficiently computing
CVaR optimal policies is an important but not well studied topic in
the MDP theory.


In the literature, most works study the CVaR of discounted
accumulated costs at a terminal stage $T$, say $\CVaR(\tilde{X}_T)$,
where $\tilde{X}_T := \sum_{t=0}^T \beta^t X_t$, $X_t$ is the cost
at time $t$, and $\beta$ is the discount factor. However, people
also care about the cost fluctuation during the procedure, i.e.,
$\CVaR(X_t)$ at each time $t$. For example, in financial
engineering, a risk-averse investor cannot tolerate high risks
during the asset management process, even though the variation of
returns at the end of the contract is small. High fluctuations of
the asset value may bring anxiety to the risk-averse investor. On
the other hand, it is becoming popular to add more inspection points
(monthly or even daily) to measure portfolio performance, rather
than a single inspection point at the end of the contract.
Therefore, the risk measure of procedure behaviors is significant
for the theory of MDPs. Compared with the existing works on the
steady variance optimization of MDPs \citep{Chung94,Sobel94,Xia20},
the steady CVaR measure in MDPs is rarely studied.

In this paper, we study a discrete-time undiscounted MDP with a
long-run CVaR criterion. Different from the CVaR metrics for
discounted costs, the long-run CVaR measures the conditional
expectation of costs when the MDP reaches steady-state. Since the
traditional approach of dynamic programming is not applicable for
such non-standard MDP problems, we study this problem from the
viewpoint of sensitivity-based optimization. By introducing a
so-called pseudo CVaR, we derive a bilevel MDP formulation with
nested structures, where the inner one is a standard MDP for
minimizing the pseudo CVaR. Then, we derive a closed-form difference
formula to quantify the difference of long-run CVaR under any two
policies. The CVaR difference formula has an elegant form which
reduces the difficulty caused by the non-additive CVaR cost
function. The optimality of deterministic policies is also shown
from the nested formulation. With the CVaR difference formula, we
further derive an optimality equation called the Bellman
\emph{local} optimality equation, which describes a necessary and
sufficient condition of local optimal policies in the mixed policy
space, while is only necessary for global optimal policies. A CVaR
derivative formula is also obtained. With the CVaR sensitivity
formulas and the Bellman local optimality equation, we develop an
iterative algorithm which behaves similarly to the classical policy
iteration algorithm and can efficiently converge to local optima.
Some interesting extensions are also discussed, including the
mean-CVaR optimization problem and the maximization problem of CVaR.
Finally, we conduct numerical examples about portfolio management to
demonstrate the efficiency of our approach.

The main contributions of this paper have three-fold. First, we
study the discrete-time undiscounted MDP with the long-run CVaR
criterion. To the best of our knowledge, our paper is the first to
investigate MDP theory for minimizing long-run CVaR. Secondly, we
derive the CVaR difference and derivative formulas, and the Bellman
local optimality equation. Thirdly, we develop a policy
iteration-type algorithm, which appears to be much more efficient
than solving the bilevel MDP problem studied in the literature
(equivalent to solving a series of standard MDPs with the number of
MDPs equal to the number of possible values of the outer-tier
parameter $y \in \mathbb R$). Along with our research direction, we
may further develop value iteration-type or policy gradient
algorithms, which can enrich the algorithmic study on CVaR dynamic
optimization.

The rest of the paper is organized as follows. In
Section~\ref{section_model}, we give the MDP formulation with the
long-run CVaR criterion. In Section~\ref{section_result}, we present
the main results of this paper, including the CVaR difference
formula, the Bellman local optimality equation, iterative
algorithms, and the related theorems. In
Section~\ref{section_extension}, we discuss some possible extensions
of our results. In Section~\ref{section_experiment}, numerical
experiments are conducted to demonstrate our main results. Finally,
we conclude this paper and discuss future research topics in
Section~\ref{section_conclusion}.

\section{Problem Formulation}\label{section_model}
Consider a discrete-time finite MDP with tuple $\langle \mathcal S,
\mathcal A, \mathcal P, c \rangle$, where the state space is
$\mathcal S := \{1,2,\dots, S\}$  and the action space is $\mathcal
A :=
\{a_1,a_2,\dots, a_A\}$. 
A function $\mathcal P : \mathcal S \times \mathcal A
\overset{D}{\mapsto} \mathcal S$ is the state transition probability
kernel with element $p(j|i,a)$, where $\overset{D}{\mapsto}$
represents a mapping to the distribution on the successor $\mathcal
S$. The element $p(j|i,a)$ indicates the transition probability to
state $j$ when action $a$ is adopted at the current state $i$.
Obviously, we have $\sum_{j \in \mathcal S} p(j|i,a) = 1$ for any
$i\in \mathcal S, a \in \mathcal A$. We denote $c : \mathcal S
\times \mathcal A \mapsto \mathbb R$ as the cost function and its
element $c(i,a)$ is the cost incurred at state $i$ with action $a$.
In this paper, we limit our discussion to stationary policies which
make the associated Markov chains always ergodic (the main results
in this paper may be extended to unichain MDPs). A function $d :
\mathcal S \times \mathcal A \mapsto [0,1]$ describes a stationary
randomized policy and its element $d(i,a)$ indicates the probability
of adopting action $a$ at state $i$, where $\sum_{a \in \mathcal A}
d(i,a) = 1$ for any $i \in \mathcal S$. If $d$ is a deterministic
policy, we also use $d(i) \in \mathcal A$ to indicate the action
adopted at state $i$, with a slight abuse of notation. We denote
$\mathcal D$ and $\mathcal D_0$ as the space of stationary
randomized and deterministic policies, respectively.

We investigate the steady-state behavior of the MDP with policy $d$.
We denote $\pi^d : \mathcal S \times \mathcal A \mapsto [0,1]$ as
the steady distribution on the space of state-action pairs. We also
denote $\pi^d(i) = \sum_{a \in \mathcal A} \pi^d(i,a)$, $i \in
\mathcal S$. The transition probability matrix of the MDP with
policy $d$ is denoted as $P^d$ whose element is $P^d(i,j) = \sum_{a
\in \mathcal A} p(j|i,a)d(i,a)$, $i,j \in \mathcal S$. By denoting
$C^d_t$ as the instantaneous cost of the MDP at time $t$, we can
view $C^d_t$ as a random variable whose distribution is determined
by the initial state distribution $\nu$ and the $t$-step transition
probability matrix $(P^d)^t$, i.e., $\nu (P^d)^t$. Obviously, we
have $\lim\limits_{t\rightarrow \infty}\nu (P^d)^t = \pi^d$ for any
$\nu$. Since the MDP is finite, $C^d_t$ is a \emph{discrete} random
variable with distribution function $F_{C^d_t}(w)$ and its $\CVaR$
at a probability level $\alpha \in (0,1)$ is defined as below.
\begin{equation}
\CVaR(C^d_t) := \mathbb E \{ C^d_t | C^d_t \geq F_{C^d_t
}^{-1}(\alpha) \}, \nonumber
\end{equation}
where $F_{C^d_t}^{-1}(\alpha)$ is the inverse distribution function
of random variable $C^d_t$, which is also known as its value at risk
(VaR) at $\alpha$, or called $\alpha$-quantile
\begin{equation}
F_{C^d_t}^{-1}(\alpha) = \VaR(C^d_t) := \inf_{w \in \mathbb R}\{w :
F_{C^d_t}(w) \geq \alpha \}. \nonumber
\end{equation}
The long-run CVaR of this MDP under policy $d$ is further defined as
\begin{equation}\label{eq_CVaRdef0}
\CVaR^d := \lim\limits_{T \rightarrow \infty} \frac{1}{T}
\sum_{t=0}^{T-1} \CVaR(C^d_t).
\end{equation}
Since the MDP is assumed always ergodic under any policy $d \in
\mathcal D$, we have
\begin{equation}
\CVaR^d = \lim\limits_{t \rightarrow \infty} \CVaR(C^d_t). \nonumber
\end{equation}
Similarly, we also have
\begin{equation}\label{eq_VaR1}
\VaR^d := \lim\limits_{t \rightarrow \infty} \VaR(C^d_t). 
\end{equation}
Our optimization objective is to find the optimal policy $d^*$ which
attains the minimum of $\CVaR^d$, i.e.,
\begin{equation}\label{eq_CVaRoptm}
\begin{array}{rcl}
d^* &=& \argmin\limits_{d \in \mathcal D}\{ \CVaR^d \},\\
\CVaR^* &=& \min\limits_{d \in \mathcal D}\{
\CVaR^d \} = \CVaR^{d^*}.\\
\end{array}
\end{equation}

This MDP optimization problem with the long-run CVaR criterion is
difficult to handle. To accommodate the CVaR metric, we define a new
cost function as
\begin{equation}\label{eq_CVaRf}
\tilde{c}(y,i,a) := y + \frac{1}{1-\alpha}[c(i,a) - y]^+, \quad
\forall i \in \mathcal S, a \in \mathcal A, y \in \mathbb R.
\end{equation}
By denoting $\tilde{c}(y) := (\tilde{c}(y,i,a): i \in \mathcal S, a
\in \mathcal A)$, we can obtain that the long-run cost of this new
MDP $\langle \mathcal S, \mathcal A, \mathcal P, \tilde{c}(y)
\rangle$ is
\begin{equation}\label{eq_pif}
\widetilde{\CVaR}^d(y) = \lim\limits_{T\rightarrow \infty}
\frac{1}{T}\sum_{t=0}^{T-1} \tilde{c}(y,X_t,A_t) = \sum_{i\in
\mathcal S, a\in \mathcal A} \pi^d(i,a)\tilde{c}(y,i,a),
\end{equation}
where $X_t$ and $A_t$ is the system state and action at time $t$,
respectively. We call $y$ and $\widetilde{\CVaR}^d(y)$ the
\emph{pseudo VaR} and the \emph{pseudo CVaR} of the MDP,
respectively. When $y$ equals the real VaR, then the pseudo CVaR
equals the real CVaR of the MDP, i.e.,
\begin{equation}\label{eq_CVaR3}
\CVaR^d = \widetilde{\CVaR}^d(y)\Big|_{y=\VaR^d}.
\end{equation}

From (\ref{eq_pif})\&(\ref{eq_CVaR3}), we can see that the long-run
CVaR optimization of MDP $\langle \mathcal S, \mathcal A, \mathcal
P, c \rangle$ is equivalent to the long-run average optimization of
MDP $\langle \mathcal S, \mathcal A, \mathcal P, \tilde{c}(\VaR^d)
\rangle$. However, the element of $\tilde{c}(\VaR^d)$, i.e.,
$\tilde{c}(\VaR^d,i,a)$, depends on not only the state-action pair
$(i,a)$, but also the whole policy behavior. That is,
$\tilde{c}(\VaR^d)$ is not an additive cost function and the tuple
$\langle \mathcal S, \mathcal A, \mathcal P, \tilde{c}(\VaR^d)
\rangle$ is not a standard MDP model. Classical MDP theory does not
apply to this long-run CVaR optimization problem.

\section{Main Results}\label{section_result}
In this section, we use the sensitivity-based optimization (SBO)
theory \citep{Cao07} to study this long-run CVaR optimization
problem.  First, we restate an important property of CVaR, which was
discovered by \cite{Rockafellar02}.
\begin{lemma}\label{lemma1}
The CVaR of random variable $X$ with probability level $\alpha$ is
equal to the minimum of the following convex optimization problem
\begin{equation}
\CVaR(X) = \min_{y \in \mathbb R}\left\{ y +
\frac{1}{1-\alpha}\mathbb E[X-y]^+ \right \},
\end{equation}
where $y^* = F_X^{-1}(\alpha)$ or its $\VaR$ attains the minimum.
\end{lemma}

With Lemma~\ref{lemma1} and \eqref{eq_CVaR3}, we have
\begin{equation}\label{eq_CVaRdef2}
\CVaR^d = \min_{y \in \mathbb R}\widetilde{\CVaR}^d(y) =
\widetilde{\CVaR}^d(y)\Big|_{y=\VaR^d}.
\end{equation}
We can convert the long-run CVaR optimization problem from a
non-standard MDP model into a \emph{bilevel MDP problem}
\begin{equation}\label{eq_2level}
\CVaR^* = \min_{d \in \mathcal D}\min_{y \in \mathbb R}
\widetilde{\CVaR}^d(y) = \min_{y \in \mathbb R}\min_{d \in \mathcal
D} \widetilde{\CVaR}^d(y).
\end{equation}

\noindent\textbf{Remark~1.} The inner problem of (\ref{eq_2level}),
solving $\min\limits_{d \in \mathcal D} \widetilde{\CVaR}^d(y)$, is
a standard MDP with tuple $\langle \mathcal S, \mathcal A, \mathcal
P, \tilde{c}(y) \rangle$ since $y$ is fixed and independent of
policy $d$. Solving the original CVaR optimization problem
(\ref{eq_CVaRoptm}) is equivalent to solving a series of MDPs
$\min\limits_{d \in \mathcal D} \widetilde{\CVaR}^d(y)$, which looks
feasible but is computationally exhaustive since the number of $y\in
\mathbb R$ is huge.

Since a long-run average MDP can be formulated as a linear program
(refer to Chapter 8.8 of \cite{Puterman94}), we can also rewrite the
bilevel MDP problem \eqref{eq_2level} as the following mathematical
program
\begin{equation}\label{eq_2levelLP}
\begin{array}{ccl}
&\min\limits_{x,\ y\in\mathbb R} & \left\{ y + \frac{1}{1-\alpha}
\sum\limits_{i \in \mathcal
S, a \in \mathcal A} x(i,a)[c(i,a)-y]^+ \right \},\\
&\mbox{s.t., } &\sum\limits_{a \in \mathcal A}x(i,a) =
\sum\limits_{j \in \mathcal S, a \in \mathcal A}x(j,a) p(i|j,a),
\quad \forall i
\in \mathcal S, \\
&&\sum\limits_{i\in \mathcal S,a \in \mathcal A}x(i,a)=1, \\
&&x(i,a) \geq 0, \quad \forall i \in \mathcal S, \ a \in \mathcal A.
\end{array}
\end{equation}
Although the constraints of \eqref{eq_2levelLP} are linear, the
objective is quadratic and not convex. Therefore,
\eqref{eq_2levelLP} is not a convex optimization problem. In fact,
as illustrated by the numerical experiment in
Section~\ref{section_experiment}, this problem may have multiple
local minima, which is difficult to solve and establishes the
non-convexity.

Fortunately, we find that it is not necessary to search $y$ in the
whole real number space $\mathbb R$. By Lemma~\ref{lemma1}, we know
that $y^*$ in \eqref{eq_CVaRdef2} equals $\VaR^d$. Thus, we define
\begin{equation}
\mathbb Y := \{\VaR^d : d \in \mathcal D \}.
\end{equation}
For obtaining $\mathbb Y$, it is exhaustive to enumerate every
possible $\VaR^d$. However, we observe that
\begin{equation}
\VaR^d = \inf\{ w \in \mathbb R : \lim\limits_{t\rightarrow
\infty}\mathbb P(c(X_t, A_t) \leq w) \geq \alpha \} \in \mathbb C,
\nonumber
\end{equation}
where the set $\mathbb C$ is composed of all possible $c(i,a)$'s,
i.e.,
\begin{equation}
\mathbb C := \{c(i,a): i \in \mathcal S, a \in \mathcal A\}.
\end{equation}
Obviously, we have $\mathbb Y \subseteq \mathbb C \subset \mathbb
R$. The number of standard MDPs in the bilevel problem
\eqref{eq_2level} is significantly reduced and we directly derive
the following lemma.
\begin{lemma}\label{lemma2bi}
The long-run CVaR optimization problem \eqref{eq_CVaRoptm} is
equivalent to the following bilevel MDP problem
\begin{equation}\label{eq_bilevel2}
\CVaR^* = \min_{y \in \mathbb C}\min_{d \in \mathcal D}
\widetilde{\CVaR}^d(y),
\end{equation}
where the number of inner standard MDP problems equals $|\mathbb C|
\leq SA$.
\end{lemma}

With Lemma~\ref{lemma2bi}, although the complexity of the long-run
CVaR MDP problem is significantly reduced, it is still
computationally intensive. In the rest of the paper, we resort to
other approaches to solve this problem more efficiently. Different
from traditional dynamic programming, the SBO theory studies the
policy optimization of Markov systems by analyzing performance
sensitivity information \citep{Cao07,Xia14}, namely the difference
formula and the derivative
formula. 

First, with Lemma~\ref{lemma2bi}, we can directly derive the
optimality of deterministic policies for the long-run CVaR MDP
problem.
\begin{theorem}\label{theorem3}
The optimum $\CVaR^*$ can be achieved by a deterministic stationary
policy.
\end{theorem}
\begin{proof}
We assume that $(y^*, d^*)$ is the optimal solution of
\eqref{eq_bilevel2}. Obviously, the inner problem $\min\limits_{d
\in \mathcal D} \widetilde{\CVaR}^d(y^*)$ is a standard MDP, and it
is well known that its optimal policy $d^*$ can be a deterministic
stationary policy \citep{Puterman94}.
\end{proof}

Therefore, we can focus only on the deterministic stationary policy
space $\mathcal D_0$. With the SBO theory, we compare the long-run
average pseudo CVaR difference of MDPs under any two policies $d$
and $d'$, where $d$ is the current policy and $d'$ is any other new
policy. We directly have the following \emph{performance difference
formula} (refer to Chapter~4.1 of \cite{Cao07})
\begin{equation} \label{eq_diff}
\widetilde{\CVaR}^{d'}(y) - \widetilde{\CVaR}^{d}(y) = \sum_{i \in
\mathcal S} \pi^{d'}(i) \left [ \sum_{j \in \mathcal S}
[p(j|i,d'(i)) - p(j|i,d(i))]g^d(y,j) + \tilde{c}(y,i,d'(i)) -
\tilde{c}(y,i,d(i)) \right],
\end{equation}
where $g^d(y)$ is a column vector called \emph{performance
potentials} whose element is defined as
\begin{equation}\label{eq_g0}
g^d(y,i) := \lim\limits_{T \rightarrow \infty} \mathbb E^d \left\{
\sum_{t=0}^{T} [\tilde{c}(y,X_t, A_t) -
\widetilde{\CVaR}^{d}(y)]\Big|X_0=i \right\}, \qquad i \in \mathcal
S.
\end{equation}
In the literature, $g^d(y)$ is also called the \emph{relative value
function} \citep{Puterman94}, which can be determined by the
following \emph{Poisson equation}
\begin{equation}\label{eq_g1}
g^d(y,i) = \tilde{c}(y,i,d(i)) - \widetilde{\CVaR}^{d}(y) + \sum_{j
\in \mathcal S} p(j|i,d(i)) g^d(y,j), \quad i \in \mathcal S.
\end{equation}
By setting $y=\VaR^d$, with \eqref{eq_CVaR3} and \eqref{eq_diff}, we
derive the CVaR difference formula
\begin{eqnarray} \label{eq_diffCVaR}
\CVaR^{d'} - \CVaR^{d} &=& \sum_{i \in \mathcal S} \pi^{d'}(i) \left
[ \sum_{j \in \mathcal S} [p(j|i,d'(i)) - p(j|i,d(i))]g^d(y,j) +
\tilde{c}(y,i,d'(i)) - \tilde{c}(y,i,d(i)) \right] \nonumber\\
&& + \CVaR^{d'} - \widetilde{\CVaR}^{d'}(y), \qquad \mbox{where }
y=\VaR^d.
\end{eqnarray}
The last term of the above equation quantifies the difference
between the real CVaR and the pseudo CVaR of policy $d'$, distorted
by policy $d$. For notation simplicity, we define
\begin{equation}\label{eq_Delta}
\Delta_{\CVaR}(d', d) := \CVaR^{d'} - \widetilde{\CVaR}^{d'}(\VaR^d)
\leq 0, \quad  \forall d,d' \in \mathcal D_0,
\end{equation}
where the inequality directly follows \eqref{eq_CVaRdef2}.
Therefore, we can rewrite \eqref{eq_diffCVaR} and derive the
following lemma about \emph{the long-run CVaR difference formula}.

\begin{lemma}\label{lemma2}
If the deterministic policy is changed from $d$ to a new policy
$d'$, where $\forall d,d' \in \mathcal D_0$, then the difference of
their long-run CVaR is quantified by
\begin{eqnarray}\label{eq_diff_CVaR2}
\CVaR^{d'} - \CVaR^{d} &=& \sum_{i \in \mathcal S} \pi^{d'}(i)
\Bigg[ \sum_{j \in \mathcal S} [p(j|i,d'(i)) -
p(j|i,d(i))]g^d(\VaR^d,j) \nonumber\\
&& + \tilde{c}(\VaR^d,i,d'(i)) - \tilde{c}(\VaR^d,i,d(i)) \Bigg] +
\Delta_{\CVaR}(d',d).
\end{eqnarray}
\end{lemma}

\noindent\textbf{Remark~2.} From the right-hand-side of the long-run
CVaR difference formula \eqref{eq_diff_CVaR2}, we can see that the
first part is the long-run average performance difference for a
standard MDP model, where the cost function $\tilde{c}(y,i,a)$ has a
constant $y=\VaR^{d}$ estimated under the current policy $d$. The
second part $\Delta_{\CVaR}(d',d)$ is computationally consuming, but
its value is always \emph{non-positive}. Therefore, we can develop
an approach to generate an improved policy $d'$ based on the above
difference formula, which is described by the following theorem.

\begin{theorem}\label{theorem1}
For the current policy $d$, if we find a new policy $d' \in \mathcal
D_0$ which satisfies
\begin{equation}\label{eq_generateD}
\sum_{j \in \mathcal S} p(j|i,d'(i)) g^{d}(\VaR^d,j) +
\tilde{c}(\VaR^d,i,d'(i)) \leq \sum_{j \in \mathcal S} p(j|i,d(i))
g^{d}(\VaR^d,j) + \tilde{c}(\VaR^d,i,d(i)), \forall i \in \mathcal
S,
\end{equation}
then we have $\CVaR^{d'} \leq \CVaR^{d}$. If the above inequality
strictly holds for at least one state $i$, then $\CVaR^{d'} <
\CVaR^{d}$.
\end{theorem}
\begin{proof}
For the two policies $d$ and $d'$, we substitute
\eqref{eq_generateD} into the difference formula
\eqref{eq_diff_CVaR2}
\begin{equation}
\CVaR^{d'} - \CVaR^{d} \leq 0 + \Delta_{\CVaR}(d',d) \leq 0,
\nonumber
\end{equation}
where the first inequality uses the fact that $\pi^{d'}(i) > 0$
since the MDP is always ergodic, and the second inequality directly
follows \eqref{eq_Delta}.

If \eqref{eq_generateD} strictly holds for at least one state $i$,
we can directly have
\begin{equation}
\CVaR^{d'} - \CVaR^{d} < 0 + \Delta_{\CVaR}(d',d) \leq 0. \nonumber
\end{equation}
Therefore, the theorem is proved.
\end{proof}

Theorem~\ref{theorem1} indicates an approach to generate improved
policies: we only need to find new policies $d'$ satisfying
\eqref{eq_generateD}, where the values of vectors $g^{d}(\VaR^d)$
and $\tilde{c}(\VaR^d)$ are computable or estimatable based on the
system sample path of the current policy $d$. One example of
generating improved policies is similar to policy improvement in
classical policy iteration:
\begin{equation}
d'(i) = \argmin\limits_{a \in \mathcal A}\left\{
\tilde{c}(\VaR^d,i,a) + \sum_{j \in \mathcal S} p(j|i,a)
g^{d}(\VaR^d,j) \right\}, \quad i \in \mathcal S. \nonumber
\end{equation}
We further derive the following theorem about the necessary
condition of optimal policies of the long-run CVaR MDP.
\begin{theorem}\label{theorem4}
The long-run CVaR optimal deterministic policy $d^*$ must satisfy
the so-called Bellman local optimality equations
\begin{equation}\label{eq_BellmanEq1}
d^*(i) = \argmin\limits_{a \in \mathcal A}\left\{
\tilde{c}(\VaR^{d^*},i,a) + \sum_{j \in \mathcal S} p(j|i,a)
g^{d^*}(\VaR^{d^*},j) \right\}, \quad i \in \mathcal S.
\end{equation}
\begin{equation}\label{eq_BellmanEq2}
g^{d^*}(\VaR^{d^*},i)  + \CVaR^{d^*} = \min\limits_{a \in \mathcal
A}\left\{ \tilde{c}(\VaR^{d^*},i,a) + \sum_{j \in \mathcal S}
p(j|i,a) g^{d^*}(\VaR^{d^*},j) \right\} , \quad i \in \mathcal S.
\end{equation}
\end{theorem}

\begin{proof}
The necessary condition \eqref{eq_BellmanEq1} can be directly proved
by \eqref{eq_generateD} in Theorem~\ref{theorem1}. Below, we prove
\eqref{eq_BellmanEq2} based on the property of performance
potentials. With the definition \eqref{eq_g0}, we have
\begin{equation}
g^{d^*}(\VaR^{d^*},i) := \lim\limits_{T \rightarrow \infty} \mathbb
E^{d^*} \left\{ \sum_{t=0}^{T} [\tilde{c}(\VaR^{d^*}, X_t, A_t) -
\CVaR^{d^*}]\Big|X_0=i \right\}, \nonumber
\end{equation}
Similar to \eqref{eq_g1}, we can also derive
\begin{equation}
g^{d^*}(\VaR^{d^*},i)  + \CVaR^{d^*} =
\tilde{c}(\VaR^{d^*},i,d^*(i)) + \sum_{j \in \mathcal S}
p(j|i,d^*(i)) g^{d^*}(\VaR^{d^*},j). \nonumber
\end{equation}
Substituting \eqref{eq_BellmanEq1} into the above equation, we can
directly derive \eqref{eq_BellmanEq2}. Therefore,
\eqref{eq_BellmanEq1} and \eqref{eq_BellmanEq2} are equivalent and
the theorem is proved.
\end{proof}

We call \eqref{eq_BellmanEq1} or \eqref{eq_BellmanEq2} \emph{the
Bellman local optimality equation} for long-run CVaR MDPs, which is
analogous to the classical Bellman optimality equation for long-run
average or discounted MDPs in the literature. However, different
from the classical Bellman optimality equation (which is necessary
and sufficient for optimal policies), \eqref{eq_BellmanEq2} is only
necessary and not sufficient for a long-run CVaR optimal policy.
This is partly because the term $\tilde{c}(\VaR^{d^*},i,a)$ in
\eqref{eq_BellmanEq2} depends on the whole policy $d^*$, while the
cost function $c(i,a)$ in the classical Bellman optimality equation
only depends on the current state and action. Another explanation is
that the CVaR difference formula \eqref{eq_diff_CVaR2} has an extra
term $\Delta_{\CVaR}(d',d)$ which does not arise in a standard MDP
model.

Nevertheless, we can still use \eqref{eq_BellmanEq1} or
\eqref{eq_BellmanEq2} to find optimal policies. However, there may
exist multiple fixed point solutions to \eqref{eq_BellmanEq2}, while
the counterpart for the classical Bellman optimality equation is
unique. These multiple solutions can be further recognized as local
optima in a mixed policy space. Below, we discuss the performance
derivative of mixed policies.

For any two deterministic policies $d, d' \in \mathcal D_0$, we
define $d^{\delta, d'}$ as a \emph{mixed policy} between $d$ and
$d'$: adopt policy $d'$ with probability $\delta$ and adopt policy
$d$ with probability $1-\delta$, where $\delta$ is the mixed
probability and $\delta \in [0,1]$. Obviously, we have $d^{0, d'} =
d$ and $d^{1, d'} = d'$. By replacing $d'$ with $d^{\delta, d'}$ in
\eqref{eq_diff_CVaR2}, we compare the long-run CVaR difference of
the MDP under policy $d$ and $d^{\delta, d'}$. For notation
simplicity, we use the superscript `$\delta$' to identify the
associated quantities of policy $d^{\delta, d'}$. Thus,
\eqref{eq_diff_CVaR2} can be rewritten as
\begin{eqnarray}\label{eq_diff_delta}
\CVaR^{\delta} - \CVaR^{d} &=& \sum_{i \in \mathcal S}
\pi^{\delta}(i) \delta \Bigg[\sum_{j \in \mathcal S}
[p(j|i,d'(i)) - p(j|i,d(i))] g^{d}(\VaR^d,j) \nonumber\\
&& + \tilde{c}(\VaR^d,i,d'(i)) - \tilde{c}(\VaR^d,i,d(i)) \Bigg] +
\Delta_{\CVaR}(\delta,d).
\end{eqnarray}
When $\delta \rightarrow 0$, we can validate that
$\Delta_{\CVaR}(\delta,d)$ goes to 0 with higher order of $\delta$.
Therefore, we can derive the following lemma about \emph{the
long-run CVaR derivative formula}, and the detailed proof can be
found in the Appendix.

\begin{lemma}\label{lemma4}
For any two deterministic policies $d, d' \in \mathcal D_0$, the
derivative of the long-run CVaR  with respect to the mixed
probability $\delta$ is
\begin{equation}
\frac{\partial \CVaR^{\delta}}{\partial \delta}\Bigg|_{\delta = 0} =
\sum_{i \in \mathcal S} \pi^{d}(i) \Bigg[\sum_{j \in \mathcal S}
[p(j|i,d'(i)) - p(j|i,d(i))] g^{d}(\VaR^d,j)  +
\tilde{c}(\VaR^d,i,d'(i)) - \tilde{c}(\VaR^d,i,d(i)) \Bigg].
\end{equation}
\end{lemma}

In the mixed policy space, we give the definition of local optimum
of the long-run CVaR MDP problem.
\begin{definition}\label{definition1}
A deterministic policy $d \in \mathcal D_0$ is local optimal, if
there exists $\epsilon >0$ such that we always have
$\CVaR^{d^{\delta,d'}} \geq \CVaR^{d}$ for any $\delta \in [0,
\epsilon]$ and $d' \in \mathcal D_0$.
\end{definition}

With Lemma~\ref{lemma4} and Definition~\ref{definition1}, we can
derive the following theorem about the local optimal policies.
\begin{theorem}\label{theorem5}
The policies $d^*$ determined by the optimality equation
\eqref{eq_BellmanEq1} or \eqref{eq_BellmanEq2} are local optima of
the long-run CVaR MDP problem in the mixed policy space.
\end{theorem}
\begin{proof}
Suppose the current policy $d^*$ satisfies the optimality equation
\eqref{eq_BellmanEq1}. We choose any new perturbed policy $d' \in
\mathcal D_0$ and study the derivative of the long-run CVaR with
respect to the mixed probability $\delta$ along with the
perturbation direction $(d^*,d')$:
\begin{eqnarray}\label{eq_49}
\frac{\partial \CVaR^{\delta}}{\partial \delta}\Bigg|_{\delta = 0}
&=& \sum_{i \in \mathcal S} \pi^{d^*}(i) \Bigg[\sum_{j \in \mathcal
S} [p(j|i,d'(i)) - p(j|i,d^*(i))] g^{d^*}(\VaR^{d^*},j) \nonumber\\
&& + \tilde{c}(\VaR^{d^*},i,d'(i)) - \tilde{c}(\VaR^{d^*},i,d^*(i))
\Bigg].
\end{eqnarray}
Since $d^*$ satisfies \eqref{eq_BellmanEq1}, we have
\begin{equation}
\tilde{c}(\VaR^{d^*},i,d^*(i)) + \sum_{j \in \mathcal
S}p(j|i,d^*(i)) g^{d^*}(\VaR^{d^*},j) \leq \tilde{c}(\VaR^{d^*},i,a)
+ \sum_{j \in \mathcal S}p(j|i,a) g^{d^*}(\VaR^{d^*},j), \quad
\forall a \in \mathcal A. \nonumber
\end{equation}
Substituting the above inequality into \eqref{eq_49}, we directly
have
\begin{equation}
\frac{\partial \CVaR^{\delta}}{\partial \delta}\Bigg|_{\delta = 0}
\geq 0, \nonumber
\end{equation}
where we use the fact that $\pi^{d^*}(i)$ is always positive. By
using the first order of the Taylor expansion of $\CVaR^{\delta}$
with respect to $\delta$, we observe that $\CVaR^{\delta}$ will
always increase along with perturbation direction $(d^*, d')$, for
any $d' \in \mathcal D_0$. Therefore, with
Definition~\ref{definition1}, we can see that the long-run CVaR at
policy $d^*$ is the local optimum in the mixed policy space.
\end{proof}

\noindent\textbf{Remark~3.} Theorem~\ref{theorem5} indicates that
the optimality equation \eqref{eq_BellmanEq1} or
\eqref{eq_BellmanEq2} is necessary and sufficient for local optimal
policies, while only necessary for global optimal policies.

With Theorems~\ref{theorem1}-\ref{theorem5}, we can develop an
iterative algorithm to find a local optimal policy, which satisfies
the Bellman local optimality equation \eqref{eq_BellmanEq1} or
\eqref{eq_BellmanEq2}. The iterative algorithm is based on the CVaR
difference formula \eqref{eq_diff_CVaR2}. Improved policies are
repeatedly generated based on Theorem~\ref{theorem1}. The details
are provided by Algorithm~\ref{algo1}.

\begin{algorithm}[htbp]
  \caption{An iterative algorithm to find local optimal policies of long-run CVaR.}\label{algo1}
  \begin{algorithmic}[1]

\State arbitrarily choose an initial policy $d^{(0)} \in \mathcal
D_0$ and set $l=0$;

\Repeat

\State \hspace{-0.8cm} policy evaluation: for the current policy
$d^{(l)}$, compute or estimate the values of $\VaR^{d^{(l)}}$ and
$g^{d^{(l)}}(\VaR^{d^{(l)}})$ based on their definitions
\eqref{eq_VaR1} and \eqref{eq_g0}, respectively;

\State \hspace{-0.8cm} policy improvement: generate a new policy
$d^{(l+1)}$ as follows:
\begin{equation}\label{eq_PIVaR}
d^{(l+1)}(i) := \argmin\limits_{a \in \mathcal A}\Bigg\{
\tilde{c}(\VaR^{d^{(l)}},i,a) + \sum_{j \in \mathcal S} p(j|i,a)
g^{d^{(l)}}(\VaR^{d^{(l)}},j) \Bigg\}, \quad i \in \mathcal S.
\end{equation}
keep $d^{(l+1)}(i) = d^{(l)}(i)$ if possible, to avoid policy
oscillations;

\State set $l := l+1$;

\Until{$d^{(l)} = d^{(l-1)}$}

\Return $d^{(l)}$.

\end{algorithmic}
\end{algorithm}

From the above algorithm procedure, we can see that
Algorithm~\ref{algo1} is similar to the classical policy iteration
algorithm in the traditional MDP theory: both use policy evaluation
and policy improvement. However, one main difference between them is
that the functional operator in \eqref{eq_PIVaR} is varied with the
instantaneous cost function $\tilde{c}(\VaR^{d^{(l)}},i,a)$. In the
policy evaluation step of Algorithm~\ref{algo1}, we have to evaluate
not only $g^{d^{(l)}}(\VaR^{d^{(l)}})$, but also $\VaR^{d^{(l)}}$,
both are related to the current policy $d^{(l)}$ and used to update
the operator in \eqref{eq_PIVaR}, while the corresponding operator
in classical policy iteration is not varied with the single period
cost function. Such difference makes the principle of classical
dynamic programming infeasible in our CVaR optimization problem.
Another difference is that Algorithm~\ref{algo1} only converges to
local optima, while the classical policy iteration converges to the
global optimum of a standard MDP. The convergence analysis of
Algorithm~\ref{algo1} is stated in the following theorem.

\begin{theorem}\label{theorem6}
Algorithm~\ref{algo1} converges to a local optimum of the long-run
CVaR minimization problem.
\end{theorem}

\begin{proof}
First, we prove the convergence of Algorithm~\ref{algo1}. By
substituting \eqref{eq_PIVaR} into \eqref{eq_diff_CVaR2}, we have
\begin{eqnarray}
\CVaR^{d^{(l+1)}} - \CVaR^{d^{(l)}} &=& \sum_{i \in \mathcal S}
\pi^{d^{(l+1)}}(i) \Bigg[\sum_{j \in \mathcal S}
[p(j|i,d^{(l+1)}(i)) -p(j|i,d^{(l)}(i))]
g^{d^{(l)}}(\VaR^{d^{(l)}},j)  \nonumber\\
&& + \tilde{c}(\VaR^{d^{(l)}},i,d^{(l+1)}(i)) -
\tilde{c}(\VaR^{d^{(l)}},i,d^{(l)}(i)) \Bigg] +
\Delta_{\CVaR}(d^{(l+1)},d^{(l)}) \nonumber\\
&<& 0 + \Delta_{\CVaR}(d^{(l+1)},d^{(l)}) \leq 0. \nonumber
\end{eqnarray}
Therefore, the long-run CVaR of newly generated policy by
\eqref{eq_PIVaR} is strictly reduced. Since the deterministic policy
space $\mathcal D_0$ is finite, Algorithm~\ref{algo1} will terminate
after a finite number of iterations. Thus, the convergence is
proved.

Next, we prove the local optimum of the convergence. When
Algorithm~\ref{algo1} stops, from the stopping rule $d^{(l)} =
d^{(l-1)}$, we can see that the converged policy $d^{(l)}$ must
satisfy the optimality equation \eqref{eq_BellmanEq1}. With
Theorem~\ref{theorem5}, the converged policy $d^{(l)}$ is a local
optimum in the mixed policy space. Thus, the theorem is proved.
\end{proof}

Although Algorithm~\ref{algo1} only converges to local optima, we
can further integrate other optimization techniques to improve its
global optimality. For example, we can choose different initial
policies such that Algorithm~\ref{algo1} may converge to different
local optima and select the best one. Other exploration techniques,
such as using a small perturbation probability away from the
improved direction or combining with evolutionary algorithms, may be
utilized. Another important topic is to efficiently estimate the key
quantities $\VaR^{d^{(l)}}$ and $g^{d^{(l)}}(\VaR^{d^{(l)}})$ based
on the system sample path under the current policy $d^{(l)}$, which
makes Algorithm~\ref{algo1} policy learning capable in the framework
of reinforcement learning \citep{Sutton18,Zhou22}.

\section{Extensions}\label{section_extension}
In this section, we discuss possible extensions of the main results
presented in Section~\ref{section_result}, which are intended as
future research topics.

\subsection{Long-Run Mean-CVaR Optimization}\label{subsection_meanCVaR}
Section~\ref{section_result} aims at minimizing the long-run CVaR of
MDPs, where the CVaR is used to measure the procedure risk. In
practice, decision makers usually optimize the return and the risk
together. Similar to the mean-variance optimization proposed by
\cite{Markowitz52}, we further study the mean-CVaR optimization in
the scenario of long-run MDPs. All the results in
Section~\ref{section_result} can be extended to the long-run
combined metric of mean and CVaR.

First, we define a combined cost function considering both the
pseudo CVaR and mean costs as below.
\begin{equation}
f_{\beta}(y,i,a) := y + \frac{1}{1-\alpha}[c(i,a) - y]^+ + \beta
c(i,a), \quad i \in \mathcal S, a \in \mathcal A, y \in \mathbb R,
\end{equation}
where $\beta \geq 0$ is a coefficient to balance the weights between
the pseudo CVaR and mean costs. Therefore, the long-run average
performance of the MDP under cost function $f_{\beta}(y)$ and policy
$d$ is
\begin{equation}
\eta^{d}(f_{\beta}(y)) = \sum_{i\in \mathcal S, a\in \mathcal A}
\pi^d(i,a)f_{\beta}(y,i,a) = \widetilde{\CVaR}^{d}(y) + \beta
\eta^{d}(c), \nonumber
\end{equation}
where $\widetilde{\CVaR}^{d}(y)$ is the pseudo CVaR defined in
\eqref{eq_pif} and $\eta^{d}(c) := \sum_{i\in \mathcal S, a\in
\mathcal A} \pi^d(i,a)c(i,a)$ is the mean cost of the MDP.
Therefore, the long-run mean-CVaR optimization problem of MDPs is
defined as below.
\begin{equation}\label{eq_mean-CVaRoptm}
\begin{array}{rcl}
d^*_{\beta} &=& \argmin\limits_{d \in \mathcal D}\{ \CVaR^d
+ \beta \eta^{d}(c)\},\\
\eta^*_{\beta} &=& \min\limits_{d \in \mathcal D}\{ \CVaR^d
+ \beta \eta^{d}(c)\} = \CVaR^{d^*_{\beta}} + \beta \eta^{d^*_{\beta}}(c),\\
\end{array}
\end{equation}
which is equivalent to the following bilevel optimization problem
\begin{eqnarray}\label{eq_mean-CVaRoptm2}
\eta^*_{\beta} &=& \min_{d \in \mathcal D} \{ \eta^{d}(f_{\beta}(y))
|_{y = \VaR^d}\} = \min_{y \in \mathbb R}\min_{d \in
\mathcal D} \{ \eta^{d}(f_{\beta}(y)) \} \nonumber\\
&=& \min_{y \in \mathbb R}\min_{d \in \mathcal D} \sum_{i \in
\mathcal S, a \in \mathcal A} \pi^{d}(i,a) f_{\beta}(y,i,a) .
\end{eqnarray}
Note that $\eta^d(-c)$ can be viewed as the average return,
$\min\limits_{d \in \mathcal D}\{ \CVaR^d + \beta \eta^{d}(c)\}$ in
\eqref{eq_mean-CVaRoptm} is equivalent to $\max\limits_{d \in
\mathcal D}\{ \beta \eta^{d}(-c) - \CVaR^d \}$ which indicates the
maximization of the mean return minus the CVaR loss risk.

\noindent\textbf{Remark~4.} Comparing the mean-CVaR optimization
problem \eqref{eq_mean-CVaRoptm2} and the CVaR minimization problem
\eqref{eq_2level}, these two problems are the same except that their
cost functions $f_{\beta}(y,i,a)$ and $\tilde{c}(y,i,a)$ are
different. The main results in Section~\ref{section_result}, such as
Theorems~\ref{theorem3}-\ref{theorem6},
Lemmas~\ref{lemma2}\&\ref{lemma4}, and Algorithm~1, are also valid
for this mean-CVaR optimization problem except that we have to
replace $\tilde{c}(y,i,a)$'s with $f_{\beta}(y,i,a)$'s and also the
associated potential functions, where we usually set $y=\VaR^d$.
More specifically, the potential function \eqref{eq_g0} in
Section~\ref{section_result} should be
\begin{equation}\label{eq_gCVaRbeta}
g^{d}_{\beta}(y,i) := \lim\limits_{T \rightarrow \infty} \mathbb E^d
\left\{ \sum_{t=0}^{T} [f_{\beta}(y, X_t, A_t) -
\widetilde{\CVaR}^{d}(y) - \beta \eta^d(c)]\Big|X_0=i \right\},
\qquad i \in \mathcal S.
\end{equation}

\subsection{Long-Run CVaR Maximization Problem}
In the previous studies, we aim at minimizing the long-run CVaR of
MDPs as a means of enforcing risk aversion. If the decision-maker is
risk seeking or we focus on extreme returns instead of losses, we
are led to study the CVaR maximization problem defined below.
\begin{equation}\label{eq_CVaRoptmax}
\begin{array}{rcl}
\overline{d}^* &=& \argmax\limits_{d \in \mathcal D}\{ \CVaR^d \},\\
\overline{\CVaR}^* &=& \max\limits_{d \in \mathcal D}\{
\CVaR^d \} = \CVaR^{\overline{d}^*}.\\
\end{array}
\end{equation}
Interestingly, we find that the CVaR maximization problem can be
converted to a maximin problem
\begin{eqnarray}\label{eq_minimax0}
\overline{\CVaR}^* &=& \max_{d \in \mathcal D} \min_{y \in \mathbb
R} \sum_{i \in \mathcal S, a \in \mathcal A} \pi^{d}(i,a)\left\{ y +
\frac{1}{1-\alpha} [c(i,a)-y]^+ \right\}.
\end{eqnarray}
Since $c(i,a)$ has a finite value set, we can further specify the
value domain $y \in \mathbb R$ to a smaller set $y \in [\underline
c, \overline c]$, where $\underline c$ and $\overline c$ is the
minimum and maximum of $c(i,a)$'s, respectively. Furthermore, by the
linear programming model \eqref{eq_2levelLP} for a long-run average
MDP, we define $\mathbb X$ as the feasible domain of variables
$x(i,a)$'s, which is determined by the linear constraints in
\eqref{eq_2levelLP}. Therefore, long-run CVaR maximization can be
rewritten as the following optimization problem
\begin{equation}\label{eq_minimax1}
\overline{\CVaR}^* = \max_{x \in \mathbb X} \min_{y \in [\underline
c, \overline c]} \sum_{i \in \mathcal S, a \in \mathcal A}
x(i,a)\left\{ y + \frac{1}{1-\alpha} [c(i,a)-y]^+ \right\}.
\end{equation}
We define
\begin{equation}
f(x,y):=\sum_{i \in \mathcal S, a \in \mathcal A} x(i,a)\left\{ y +
\frac{1}{1-\alpha} [c(i,a)-y]^+ \right\}. \nonumber
\end{equation}
It is easy to verify that $f(x,y)$ is concave (actually linear) in
$x$ and convex in $y$ (refer to \citep{Rockafellar02}). Moreover, it
is easy to verify that $\mathbb X$ and $[\underline c, \overline c]$
are both compact convex sets. Therefore, the \emph{von Neumann's
minimax theorem} \citep{VonNeumann28} establishes that the operators
max and min in \eqref{eq_minimax1} are interchangeable, i.e.,
\begin{equation}
\overline{\CVaR}^* = \max_{x \in \mathbb X} \min_{y \in [\underline
c, \overline c]} f(x,y) = \min_{y \in [\underline c, \overline c]}
\max_{x \in \mathbb X}  f(x,y). \nonumber
\end{equation}
We directly derive the following lemma.
\begin{lemma}
The von Neumann's minimax theorem is applicable to the long-run CVaR
maximization problem \eqref{eq_minimax1} whose solution is the
saddle point of the concave-convex function $f(x,y)$.
\end{lemma}

\noindent\textbf{Remark~5.} Since the minimax theorem is widely used
in game theory, we may analogously use the algorithms of game theory
to solve this long-run CVaR maximization problem. For example, the
linear programming method for solving matrix games \citep{Barron08}
and the minimax-Q learning algorithm \citep{Littman94} for solving
zero-sum stochastic games may be utilized to solve
\eqref{eq_minimax1}, which is an interesting algorithmic topic
deserving further investigations.

%


\section{Numerical Experiments}\label{section_experiment}
In this section, we use a simplified portfolio management problem to
demonstrate the main results of this paper.

\subsection{CVaR Optimization of System Costs}
We consider a simplified financial market, where the initial total
wealth of an investor is $\$10,000$. For simplicity, we consider a
portfolio which consists of a risky asset and a riskless asset. The
riskless asset has a fixed rate of daily return as $r_f = 0.01\%$.
It is common to assume different market conditions, such as bull
(good) or bear (bad) market. We denote the market condition at time
$t$ as $e_t$. The market condition transition is characterized by a
transition probability $p(e_t,e_{t+1})$, which is endogenous and
independent of the investor's action. In this paper, we assume that
the market has 10 different conditions and the transition
probability matrix is given in Table~\ref{tab P}.

\begin{table}[htbp]
\small \centering
\begin{tabular}{c|cccccccccc}
\toprule \backslashbox{$e_t$}{$e_{t+1}$} & $0$  & $1$  & $2$  & $3$
& $4$  & $5$  & $6$  & $7$  & $8$  & $9$  \\ \midrule $0$  & 0.20 &
0.13 & 0.19 & 0.09 & 0.12 & 0.06 & 0.12 & 0.04 & 0.04 & 0.01\\
\hline $1$  & 0.18 & 0.15 & 0.15 & 0.09 & 0.08 & 0.15 & 0.06 & 0.07
& 0.04 & 0.03 \\ \hline $2$  & 0.13 & 0.09 & 0.12 & 0.22 & 0.14 &
0.14 & 0.04 & 0.03 & 0.07 & 0.02 \\ \hline $3$  & 0.11 & 0.10 & 0.13
& 0.12 & 0.11 & 0.15 & 0.07 & 0.08 & 0.07 & 0.06 \\ \hline $4$  &
0.07 & 0.14 & 0.15 & 0.10 & 0.13 & 0.11 & 0.11 & 0.05 & 0.07 & 0.07
\\ \hline $5$  & 0.07 &
0.09 & 0.08 & 0.06 & 0.06 & 0.18 & 0.14 & 0.14 & 0.07 & 0.11 \\
\hline $6$  & 0.08 & 0.05 & 0.13 & 0.16 & 0.11 & 0.10 & 0.11 & 0.07
& 0.09 & 0.10 \\ \hline $7$  & 0.09 & 0.06 & 0.08 & 0.16 & 0.10 &
0.07 & 0.11  & 0.13 & 0.08 & 0.12 \\ \hline $8$  & 0.07 & 0.09 &
0.07 & 0.08 & 0.13 & 0.08 & 0.12 & 0.09 & 0.13 & 0.14 \\ \hline $9$
& 0.01 & 0.15 & 0.11 & 0.08 & 0.04 & 0.15 & 0.10 & 0.11 & 0.03 &
0.22
\\ \bottomrule
\end{tabular}
\caption{The transition probability matrix of market
conditions.}\label{tab P}
\end{table}

At each time epoch $t$, the investor has to choose an action $a_t$
which determines the percentage of the risky asset at the next time
epoch. We denote $w_t$ as the percentage of the risky asset at time
$t$. Thus, we always have $w_{t+1} = a_t$. The set of possible
values of $a_t$ is $\mathcal A = \{0.1, 0.25, 0.4, 0.55, 0.7,
0.85\}$. The system state $s_t$ is composed of the market condition
$e_t$ and the percentage of the risky asset $w_t$, i.e., $s_t :=
(e_t,w_t)$. Thus, the size of the state space is $|\mathcal S| =
60$.  Under different market conditions, the risky asset has
different return rates $r_{risky}$, as shown in Table~\ref{tab R}. A
negative value of $r_{risky}$ means that there will be losses if the
investor holds the risky asset at the corresponding market
condition. If the investor chooses $a_t$ at state $s_t$, the system
will transit to state $s_{t+1}$ at the next time epoch, and a reward
$r(s_t,a_t,s_{t+1})$ will be incurred. We also consider a
transaction cost rate which is set as $b=0.45\%$. Hence, the value
of $r(s_t,a_t,s_{t+1})$ can be calculated as below.
\begin{equation}
r(s_t,a_t,s_{t+1}) = r((e_t,w_t),a_t,(e_{t+1},w_{t+1})) =
[r_{risky}(e_{t+1})  a_t - b |a_t - w_t| + r_f (1 - a_t)]\times
10^4. \nonumber
\end{equation}
For simplicity of calculation, we always reset the asset amount as
$10^4$ at each time epoch. We can further rewrite the above
instantaneous reward as $r(s_t,a_t)$ by using state transition
probabilities.
\begin{equation}
r(s_t,a_t) = \sum_{e_{t+1}=0}^{9} r(s_t,a_t,s_{t+1})p(e_t, e_{t+1}).
\nonumber
\end{equation}

\begin{table}[htbp]
\small \centering
\begin{tabular}{c|cccccccccc}
\toprule $e_t$  & 0 & 1 & 2 & 3 & 4 & 5 & 6 & 7 & 8 & 9\\
\toprule $r_{risky}$  & 0.09 & 0.08 & 0.06 & 0.05 & 0.04 & 0.03 & 0.02 & -0.001 & -0.002 & -0.05\\
\bottomrule
\end{tabular}
\caption{The return rate of the risky asset under different market
conditions.}\label{tab R}
\end{table}

For convenience, we transform the reward into cost as $c_t=-r_t$.
The probability level of CVaR is set as $\alpha = 0.66$. The
illustrative diagram of this portfolio management problem is shown
in Figure~\ref{fig_illustrative}. We first use Algorithm~1 to
minimize the long-run $\CVaR$ of this portfolio management
problem. 

\begin{figure}[htbp]
\centering
\includegraphics[width=.6\columnwidth]{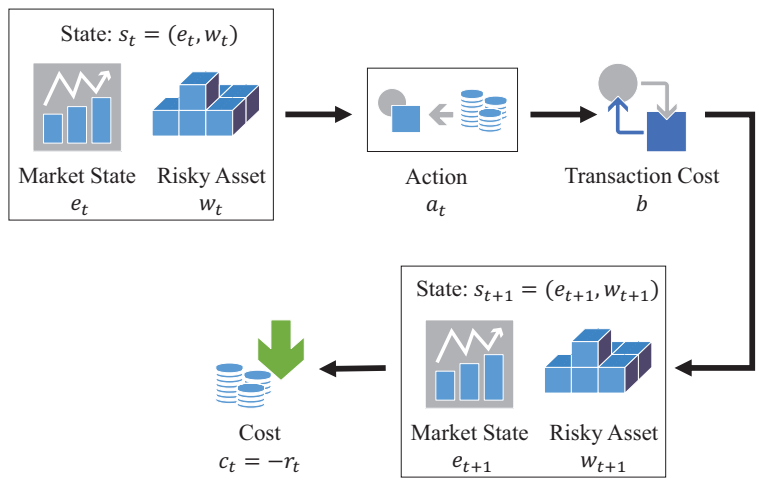}
\caption{Illustration of the portfolio management
problem.}\label{fig_illustrative}
\end{figure}

\begin{figure}[htbp]

\subfigure[The CVaR global optimum converged]{
\begin{minipage}{0.5\textwidth}
\centering
\includegraphics[width=0.9\columnwidth]{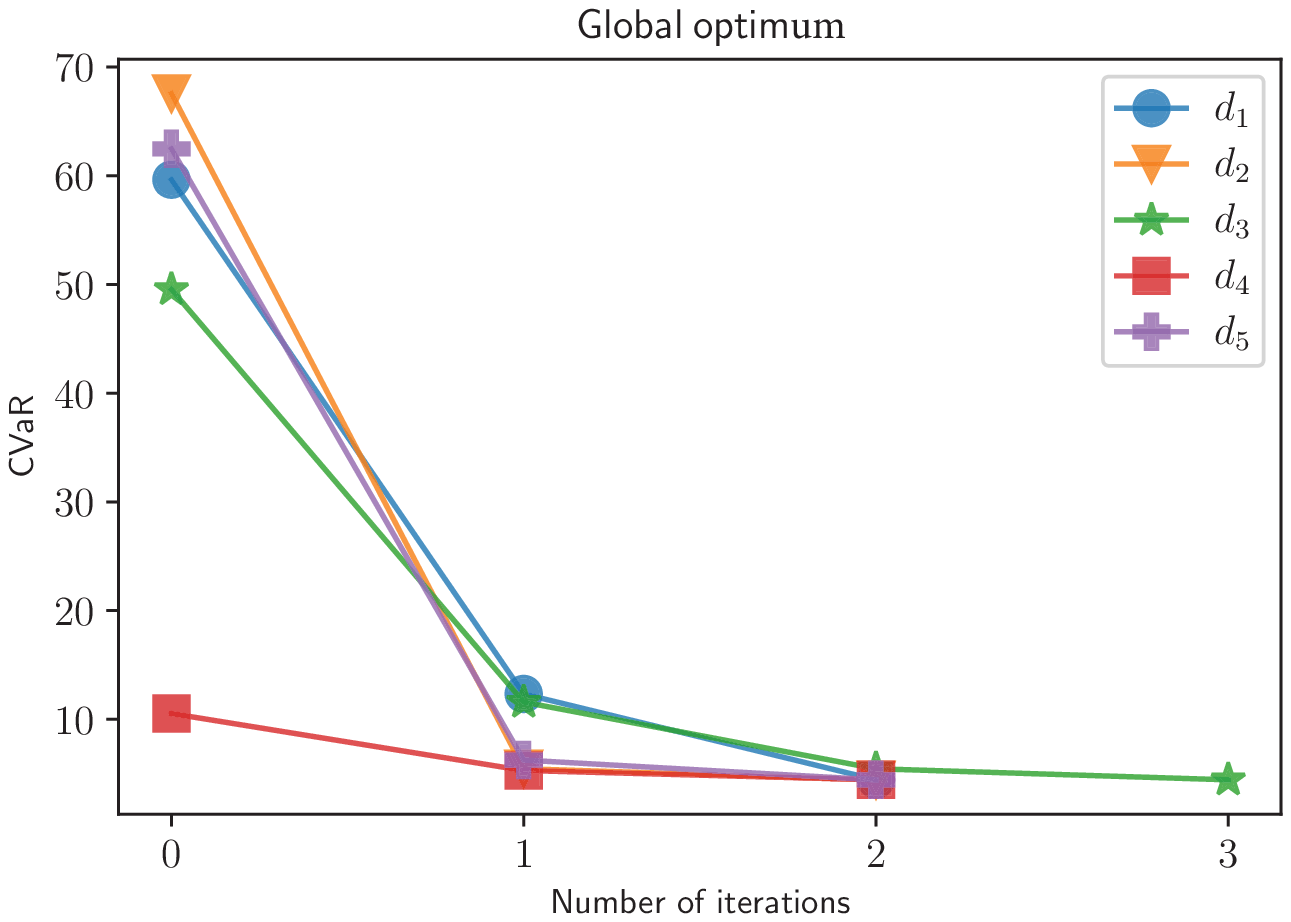}\label{subfig_CVaR_global}
\end{minipage}
} \hspace{-.3in} \subfigure[The CVaR local optimum converged]{
\begin{minipage}{0.5\textwidth}
\centering
\includegraphics[width=0.9\columnwidth]{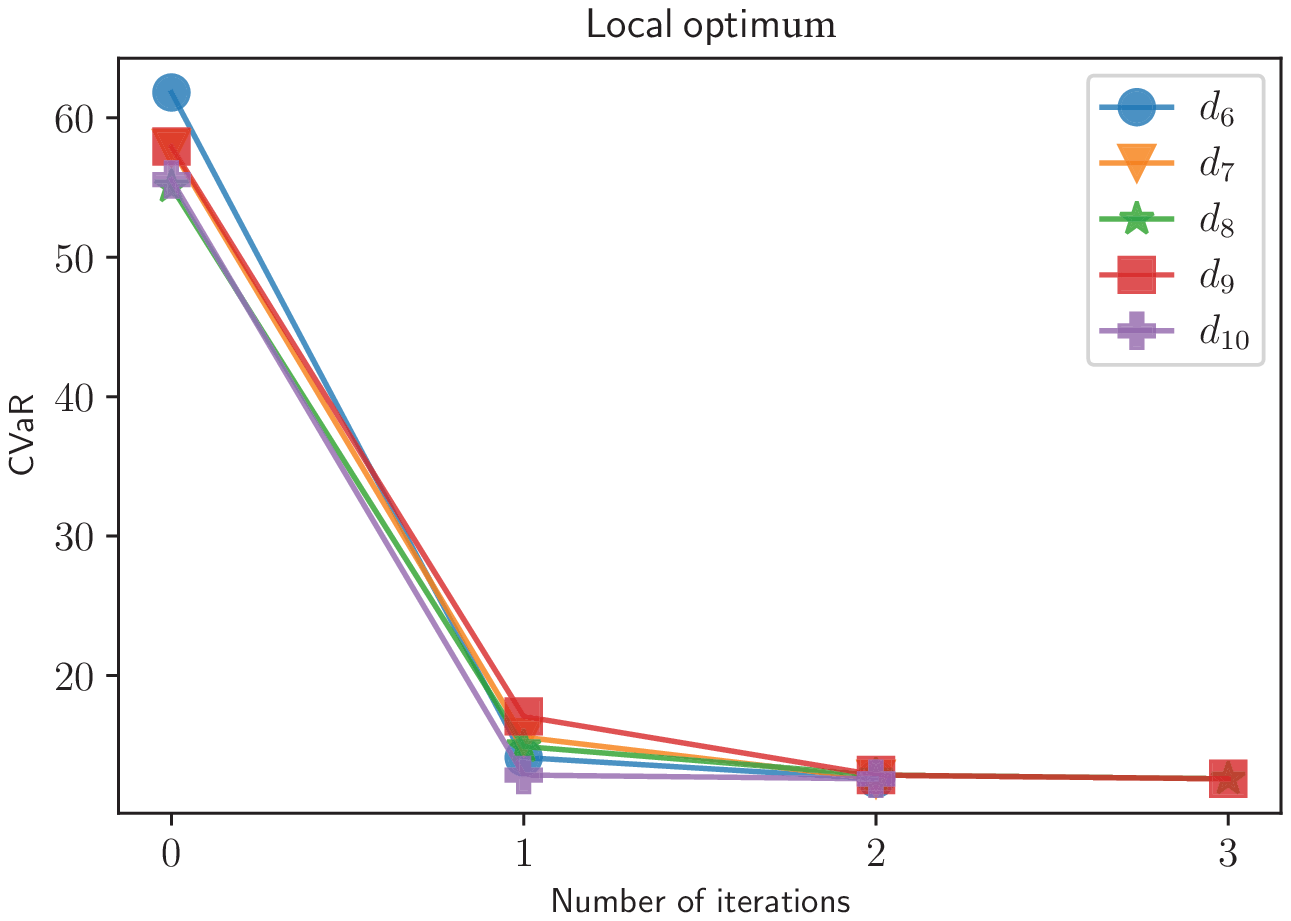}\label{subfig_CVaR_local}
\end{minipage}
}

\caption{The convergence procedure of Algorithm~\ref{algo1} under
different initial policies. The last points on the curves are the
convergence points.} \label{fig_CVaR_converge}
\end{figure}

We randomly select initial policies from the policy space and show
the experimental result as example. From
Figure~\ref{fig_CVaR_converge}, we observe that the value of CVaR
decreases strictly in the iteration process until it converges. With
different initial policies randomly chosen from the policy space,
CVaR always converges to one of the two local optima, as shown in
Figures~\ref{fig_CVaR_converge}\subref{subfig_CVaR_global}\&\subref{subfig_CVaR_local}.
For example, with initial policy $d_1$, the algorithm converges to
the global optimum at the second iteration, and the associated CVaR
is $4.43$. With initial policy $d_6$, the algorithm converges to the
local optimum at the second iteration, and the associated CVaR is
$12.58$. Given different initial policies, Algorithm~\ref{algo1}
converges within two or three iterations in most cases, which
demonstrates its fast convergence speed.

Moreover, we compare the distributions of system losses under the
CVaR global optimal policy, the CVaR local optimal policy, and a
random investment policy, which are illustrated with different
colored histograms in Figure~\ref{fig_pdf}. The policy matrix of the
CVaR global and local optimum is shown in
Figure~\ref{fig_policy_matrix}\subref{fig_policy_matrix_global}\&\subref{fig_policy_matrix_local},
respectively. The random investment policy refers to a naive one
choosing actions randomly at each state, which is illustrated by the
policy matrix in
Figure~\ref{fig_policy_matrix}\subref{fig_policy_matrix_random}.
Besides, we compute the long-run average optimal policy under which
the mean of $c_t$ attains minimum, whose policy matrix is
illustrated by
Figure~\ref{fig_policy_matrix}\subref{fig_policy_matrix_mean}. From
Figure~\ref{fig_pdf}, we can observe that the loss distribution
under the CVaR global optimal policy has the minimal expected tail
losses. The distribution of tail losses under the CVaR local optimal
policy is also thinner than that of the random investment policy. In
fact, the distribution of the CVaR global optimal policy is thinner
than that of other policies. With the CVaR global policy shown in
Figure~\ref{fig_policy_matrix}\subref{fig_policy_matrix_global}, the
investor should hold as few risky assets as possible in bear market
and hold a little more risky assets in bull market. The mean optimal
policy ignores the risk completely, and suggests the investor to
hold as more risky assets as possible in any market state, even if
there may be extreme losses. The performance comparison of mean,
standard deviation, and CVaR under different policies are summarized
in Table~\ref{tab num_result}.

\begin{figure}[htbp]
\centering
\includegraphics[width=0.6\columnwidth]{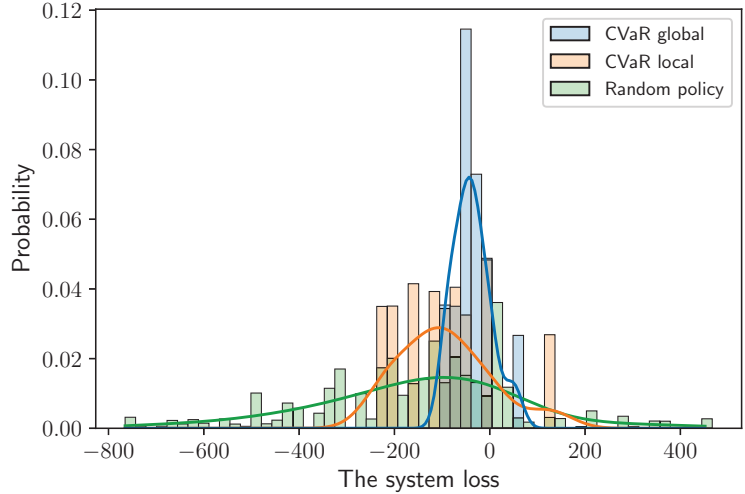}
\caption{The distributions and fitting curves of the system losses
under different policies.}\label{fig_pdf}
\end{figure}

\begin{figure*}[htbp]

\subfigure[The CVaR global optimal policy matrix]{
\begin{minipage}{0.5\textwidth}
\centering
\includegraphics[width=0.9\columnwidth]{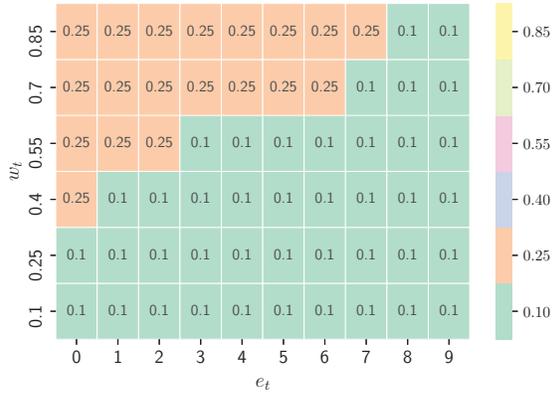}\label{fig_policy_matrix_global}
\end{minipage}
} \hspace{-.3in} \subfigure[The CVaR local optimal policy matrix]{
\begin{minipage}{0.5\textwidth}
\centering
\includegraphics[width=0.9\columnwidth]{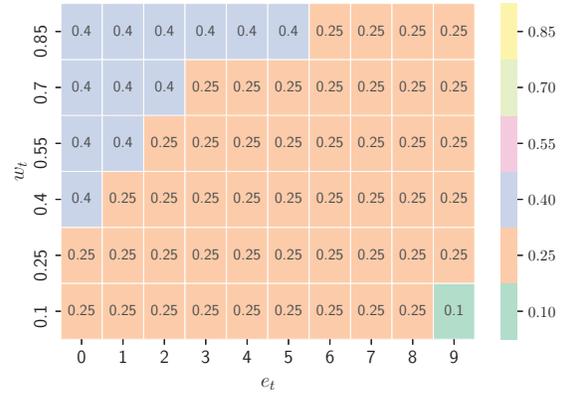}\label{fig_policy_matrix_local}
\end{minipage}
}

\subfigure[The random policy matrix]{
\begin{minipage}{0.5\textwidth}
\centering
\includegraphics[width=0.9\columnwidth]{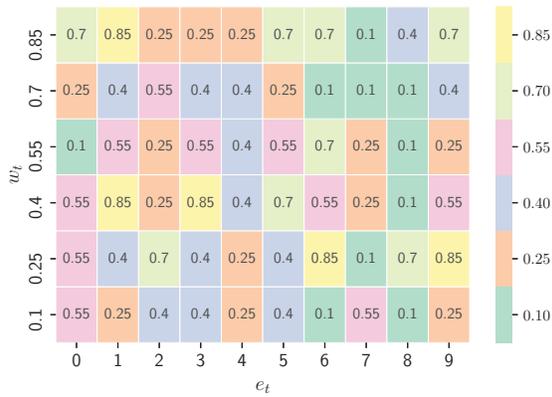}\label{fig_policy_matrix_random}
\end{minipage}
} \hspace{-.3in} \subfigure[The mean optimal policy matrix]{
\begin{minipage}{0.5\textwidth}
\centering
\includegraphics[width=0.9\columnwidth]{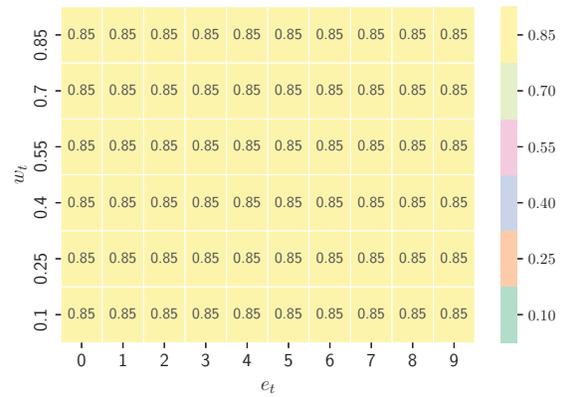}\label{fig_policy_matrix_mean}
\end{minipage}
} \caption{The policy matrices, where the horizontal axis represents
the market condition, the vertical axis represents the percentage of
the risky asset, and the number and color in the matrix represent
the value of action $a_t \in \{0.1, 0.25, 0.4, 0.55, 0.7, 0.85\}$.}
\label{fig_policy_matrix}
\end{figure*}

\begin{table}[htbp]
\small \centering
\begin{tabular}{c|ccc}
\toprule Policy  & $\eta(c_t)$ & $\sigma(c_t)$ & $\CVaR(c_t)$\\
\midrule The CVaR global optimal policy & -37.55 & 37.91 & 4.43\\
\hline   The CVaR local optimal policy & -92.37 & 94.77 & 12.58\\
\hline   The random policy  & -150.41 & 205.21 & 50.38\\
\hline   The mean optimal policy  & -311.65 & 322.20 & 45.17\\
\bottomrule
\end{tabular}
\caption{Performance comparisons of different policies.}\label{tab
num_result}
\end{table}

From the experiment results, we can see that our Algorithm~1 can
optimize the CVaR metric effectively. The optimization algorithm has
a fast convergence speed, which is analogous to the classical policy
iteration algorithm. The local convergence of Algorithm~\ref{algo1}
is also demonstrated. 

\newpage

\subsection{Mean-CVaR Optimization of System Costs}
In financial market, investors often need to control both the risk
and the mean of losses in order to obtain stable returns. Thus, we
conduct an experiment to demonstrate the mean-CVaR optimization
discussed in Section~\ref{subsection_meanCVaR}. The optimization
objective is $\CVaR(c_t) + \beta \eta(c_t)$. The probability level
of $\CVaR$ is set as $\alpha = 0.75$. All the other parameters are
the same as those in the previous experiment. We first set
$\beta=0.22$ and adapt Algorithm~\ref{algo1} to optimize the
mean-CVaR combined objective. The convergence procedure is
illustrated in Figure~\ref{fig_metric}, where we can observe that
the combined objective decreases strictly during the procedure and
converges to $3.38$ at the third iteration. 

\begin{figure}[htbp]
\centering
\includegraphics[width=0.6\columnwidth]{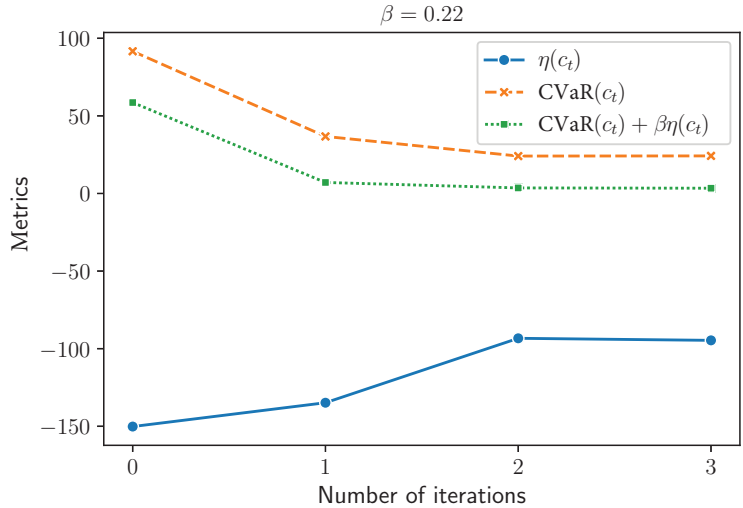}
\caption{The convergence procedure of mean-CVaR optimization in
portfolio management.}\label{fig_metric}
\end{figure}

We also investigate the optimization results under different initial
policies and different values of $\beta$. The convergence procedure
is illustrated in Figure~\ref{fig_mean_cvar}. For different values
of $\beta$, we randomly choose 5 different initial policies to
execute Algorithm~\ref{algo1}. The convergence procedures under
different initial policies are distinguished by different colored
curves. There are two local optima when $\beta = 0.4$, as shown in
Figure~\ref{fig_mean_cvar}\subref{fig_mean_cvar0.4}. In other cases,
different initial policies always converge to a single local
optimum. The value of $\beta$ reflects the risk attitude of
investors. For example, when $\beta = 0.1$, it means that the
investor is more risk averse and prefers to sacrifice potential
returns to obtain a conservative investment policy. In contrast,
when $\beta=2$, it indicates that the investor is more risk seeking
and pursues more returns by buying more risky assets. Note that
minimizing $\eta(c_t)$ is equivalent to maximizing investment
returns in this experiment. The corresponding optimal policies under
different values of $\beta$ are shown in
Figure~\ref{fig_beta_policy_matrix}, where we observe that the
optimal policy becomes more risk seeking when the value of $\beta$
increases. Performance comparisons of different mean-CVaR optimal
policies are summarized in Table~\ref{tab num_result2}. Numerical
results demonstrate that Algorithm~\ref{algo1} can be extended to
optimize both the mean and CVaR simultaneously, which is a common
application scenario in finance.

\begin{figure*}[htbp]

\subfigure[]{
\begin{minipage}{0.6\textwidth}
\includegraphics[width=0.8\columnwidth]{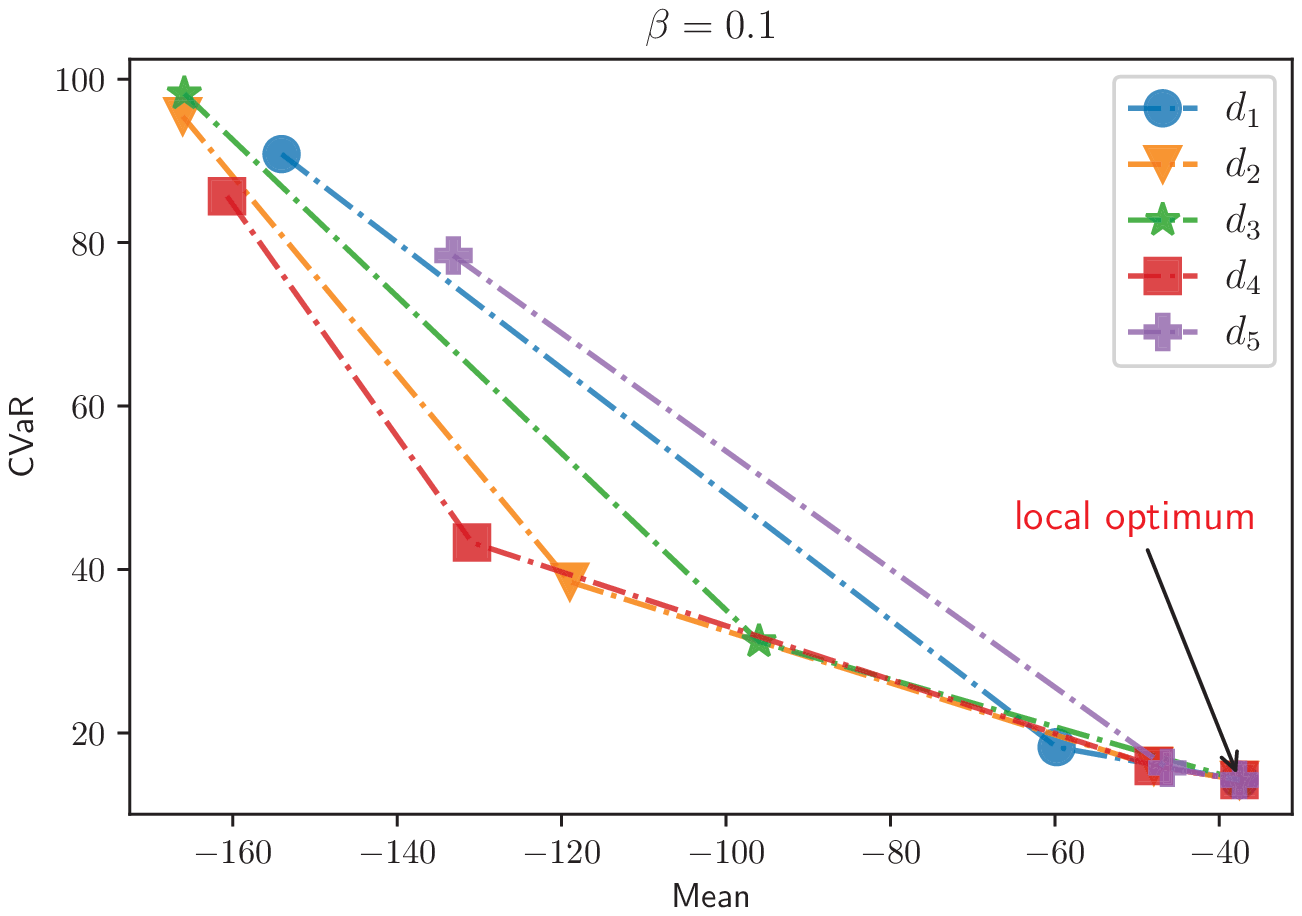}\label{fig_mean_cvar0.1}
\end{minipage}
} \hspace{-.9in} \subfigure[]{
\begin{minipage}{0.6\textwidth}
\includegraphics[width=0.8\columnwidth]{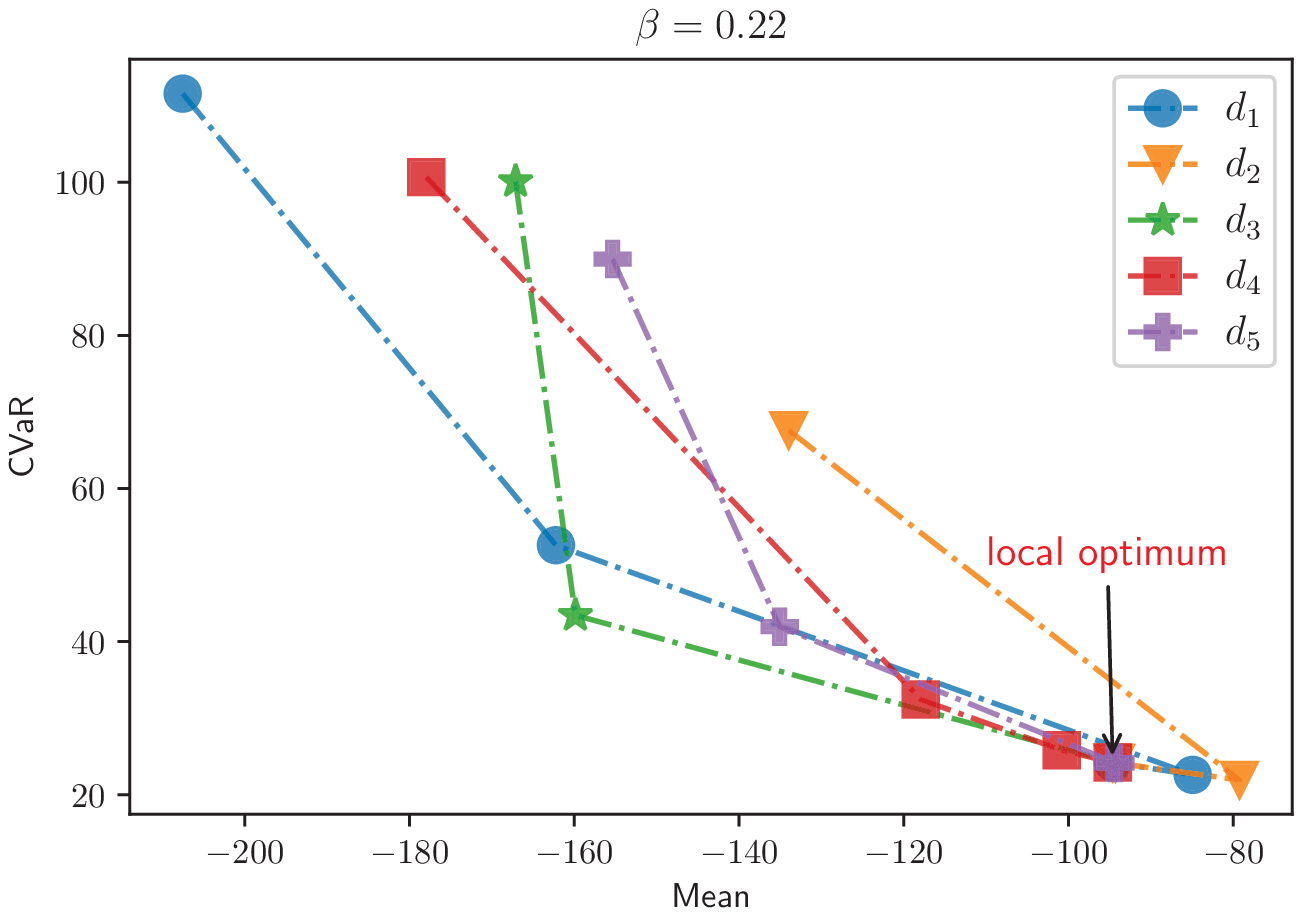}\label{fig_mean_cvar0.22}
\end{minipage}
}

\subfigure[]{
\begin{minipage}{0.6\textwidth}
\includegraphics[width=0.8\columnwidth]{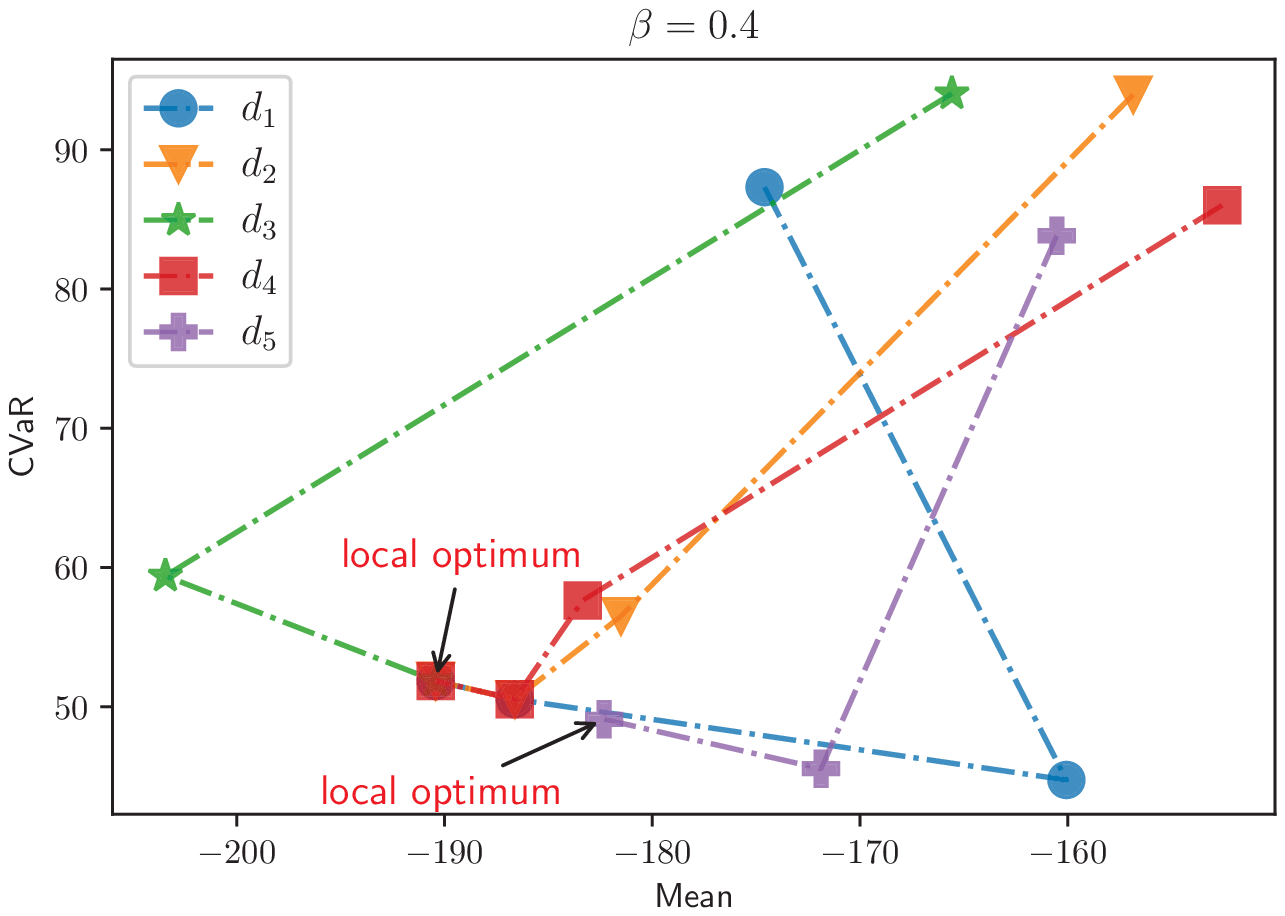}\label{fig_mean_cvar0.4}
\end{minipage}
} \hspace{-.9in} \subfigure[]{
\begin{minipage}{0.6\textwidth}
\includegraphics[width=0.8\columnwidth]{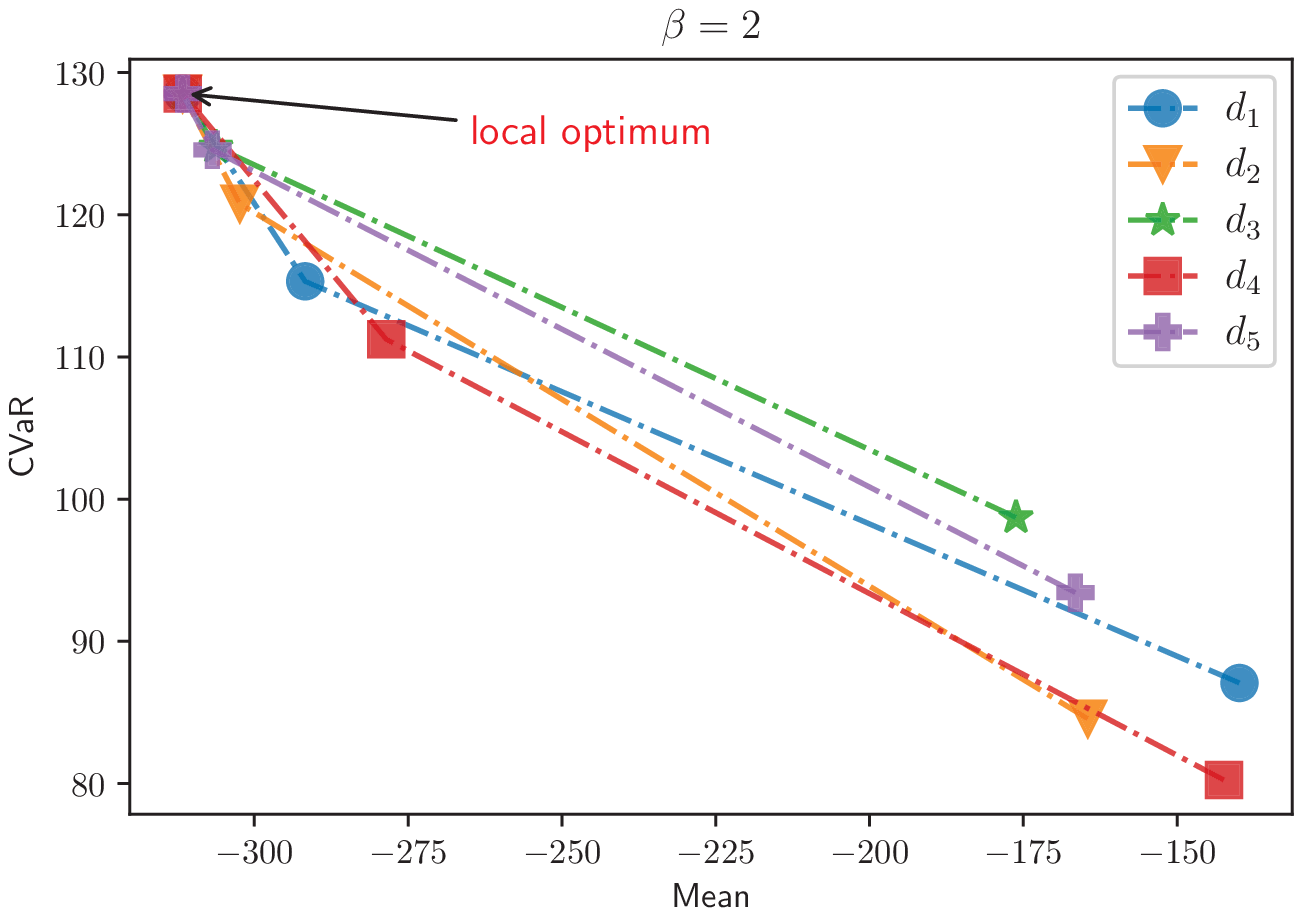}\label{fig_mean_cvar2}
\end{minipage}
}

\caption{The convergence results under different values of $\beta$
and different initial policies.} \label{fig_mean_cvar}
\end{figure*}

\begin{figure*}[htbp]

\subfigure[The optimal policy when $\beta = 0.1$]{
\begin{minipage}{0.5\textwidth}
\centering
\includegraphics[width=0.9\columnwidth]{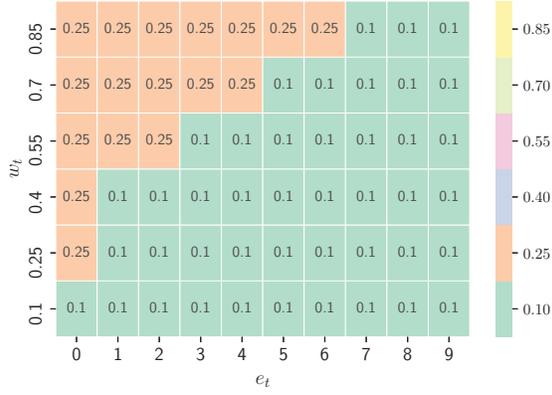}\label{fig_beta0.1_policy_matrix}
\end{minipage}
} \hspace{-.3in} \subfigure[The optimal policy when $\beta = 0.22$]{
\begin{minipage}{0.5\textwidth}
\centering
\includegraphics[width=0.9\columnwidth]{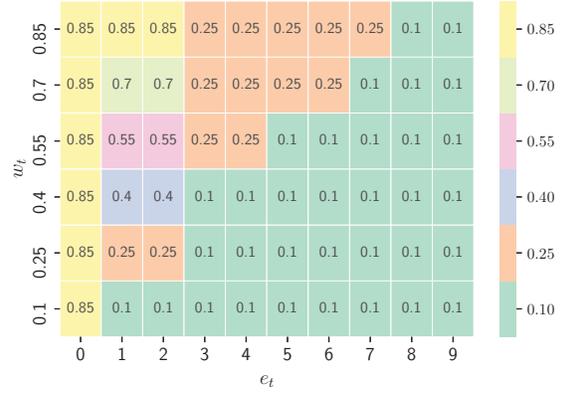}\label{fig_beta0.22_policy_matrix}
\end{minipage}
}

\subfigure[The optimal policy when $\beta = 0.4$]{
\begin{minipage}{0.5\textwidth}
\centering
\includegraphics[width=0.9\columnwidth]{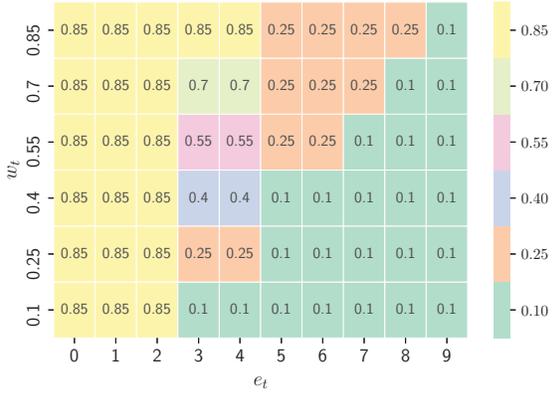}
\label{fig_beta0.4_global_policy_matrix}
\end{minipage}
} \hspace{-.3in} \subfigure[The optimal policy when $\beta = 2$]{
\begin{minipage}{0.5\textwidth}
\centering
\includegraphics[width=0.9\columnwidth]{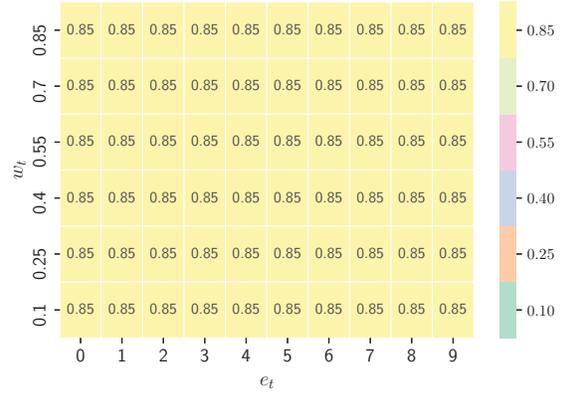}\label{fig_beta2_policy_matrix}
\end{minipage}


}

\caption{The mean-CVaR optimal policy matrices under different
values of $\beta$.} \label{fig_beta_policy_matrix}
\end{figure*}

\begin{table}[htbp]
\small \centering
\begin{tabular}{c|ccc}
\toprule Policy  & $\CVaR(c_t)$ & $\eta(c_t)$ & $\CVaR(c_t) + \beta\eta(c_t)$\\
\midrule The optimal policy when $\beta = 0.1$ & 14.24 & -37.55 & 10.48\\
\hline   The optimal policy when $\beta = 0.22$ & 24.20 & -94.64 & 3.38\\
\hline   The global optimal policy when $\beta = 0.4$  & 51.84 & -190.42 & -24.33\\
\hline   The local optimal policy when $\beta = 0.4$  & 49.09 & -182.31 & -23.84\\
\hline   The optimal policy when $\beta = 2$  & 128.52 & -311.65 & -494.77\\
\bottomrule
\end{tabular}
\caption{Performance comparisons of mean-CVaR optimal policies under
different values of $\beta$.}\label{tab num_result2}
\end{table}

\section{Discussion and Conclusion} \label{section_conclusion}
In this paper, we study the long-run CVaR optimization in the
framework of MDPs. Because of the non-additive CVaR cost function,
this dynamic optimization problem is not a standard MDP and
traditional dynamic programming is not applicable. We study this
problem by using sensitivity-based optimization. The long-run CVaR
difference formula and derivative formula are both derived.
Moreover, a so-called Bellman local optimality equation and other
properties of optimal policies are obtained. With the CVaR
sensitivity formulas, we further develop a policy iteration type
algorithm to minimize long-run CVaR, and its local convergence is
proved. Some possible extensions, including the mean-CVaR
optimization and the long-run CVaR maximization, are also discussed.

Along with this research direction, it is valuable to further study
other research topics. One topic is to extend the single objective
of CVaR to multiple objectives, such as CVaR, mean, or their ratio
metrics. Another topic is to extend our policy iteration type
algorithm to other forms, such as value iteration or policy gradient
type algorithms which are widely used in reinforcement learning. The
combination with the techniques of sampling efficiency and neural
network approximation for model-free scenarios is also a promising
direction. It is expected that the Bellman local optimality equation
and the CVaR sensitivity formulas will play important roles in the
study of these topics.


\newpage

\section*{Appendix}
\textbf{Proof of Lemma~\ref{lemma4}:}

With the CVaR difference formula \eqref{eq_diff_delta} for
deterministic policy $d$ and mixed policy $d^{\delta, d'}$, we can
see that the derivative is written as below.
\begin{eqnarray}\label{eq_56}
\frac{\partial \CVaR^{\delta}}{\partial \delta}\Big|_{\delta = 0}
&=& \lim\limits_{\delta \rightarrow 0}\frac{\CVaR^{\delta} -
\CVaR^{d}}{\delta}
\nonumber\\
&=& \lim\limits_{\delta \rightarrow 0} \sum_{i \in \mathcal S}
\pi^{\delta}(i) \Bigg[\sum_{j \in \mathcal S}
[p(j|i,d'(i)) - p(j|i,d(i))] g^{d}(\VaR^d,j) + \tilde{c}(\VaR^d,i,d'(i)) \nonumber\\
&&  - \tilde{c}(\VaR^d,i,d(i)) \Bigg] + \lim\limits_{\delta
\rightarrow 0} \frac{\Delta_{\CVaR}(\delta,d)}{\delta} \nonumber\\
&=& \sum_{i \in \mathcal S} \pi^{d}(i) \Bigg[\sum_{j \in \mathcal S}
[p(j|i,d'(i)) - p(j|i,d(i))] g^{d}(\VaR^d, j) + \tilde{c}(\VaR^d,i,d'(i)) \nonumber\\
&&  - \tilde{c}(\VaR^d,i,d(i)) \Bigg] + \lim\limits_{\delta
\rightarrow 0} \frac{\Delta_{\CVaR}(\delta,d)}{\delta},
\end{eqnarray}
where we use the fact that $\lim\limits_{\delta \rightarrow 0}
\pi^{\delta}(i) = \pi^{d}(i)$ since $d^{\delta, d'}|_{\delta = 0} =
d$. Below, we study the last term in the above equation. From the
definition \eqref{eq_Delta}, we have
\begin{equation}
\Delta_{\CVaR}(\delta,d) = \CVaR^{\delta} -
\widetilde{\CVaR}^{\delta}(\VaR^d) = \sum_{i \in \mathcal S} \sum_{a
\in \mathcal A} \pi^{\delta}(i,a) [\tilde{c}(\VaR^{\delta},i,a) -
\tilde{c}(\VaR^d,i,a)]. \nonumber
\end{equation}
Substituting \eqref{eq_CVaRf} into the above equation, we have
\begin{equation}\label{eq_58}
\Delta_{\CVaR}(\delta,d) = \VaR^{\delta} - \VaR^{d} +
\frac{1}{1-\alpha} \sum_{i \in \mathcal S} \sum_{a \in \mathcal A}
\pi^{\delta}(i,a)\Big\{ [c(i,a)-\VaR^{\delta}]^+ -
[c(i,a)-\VaR^{d}]^+\Big\}.
\end{equation}
Without loss of generality, we assume $\VaR^{\delta} > \VaR^{d}$. We
further define the following sets of state-action pairs
\begin{equation}\label{eq_59}
\begin{array}{lll}
\mathcal H^+ &:=& \{(i,a) : c(i,a) \geq \VaR^{\delta} \}, \\
\mathcal H_- &:=& \{(i,a) : c(i,a) \leq \VaR^{d} \}, \\
\mathcal H^+_- &:=& \{(i,a) : \VaR^{d} < c(i,a) < \VaR^{\delta} \}.
\end{array}
\end{equation}
Substituting \eqref{eq_59} into \eqref{eq_58}, we have
\begin{eqnarray}\label{eq_60}
\Delta_{\CVaR}(\delta,d) &=& \VaR^{\delta} - \VaR^{d} +
\frac{1}{1-\alpha} \Bigg\{ \sum_{i,a \in \mathcal H^+}
\pi^{\delta}(i,a) [c(i,a)-\VaR^{\delta} -
c(i,a) + \VaR^{d}] \nonumber\\
&& + \sum_{i,a \in \mathcal H^+_-} \pi^{\delta}(i,a) [\VaR^{\delta}
- c(i,a) - c(i,a) + \VaR^{d}] + \sum_{i,a \in \mathcal H_-}
\pi^{\delta}(i,a) [0 -
0] \Bigg \} \nonumber\\
&=& \VaR^{\delta} - \VaR^{d} - \frac{\VaR^{\delta} -
\VaR^{d}}{1-\alpha} \sum_{i,a \in \mathcal H^+} \pi^{\delta}(i,a) +
\frac{1}{1-\alpha} \sum_{i,a \in \mathcal H^+_-} \pi^{\delta}(i,a)
[\VaR^{\delta} \nonumber\\
&& - 2c(i,a) + \VaR^{d}] \nonumber\\
&=&  \frac{1}{1-\alpha} \sum_{i,a \in \mathcal H^+_-}
\pi^{\delta}(i,a) [\VaR^{\delta} - 2c(i,a) + \VaR^{d}],
\end{eqnarray}
where the last equality utilizes the fact that $\sum\limits_{i,a \in
\mathcal H^+} \pi^{\delta}(i,a) = 1 - \alpha$. From \eqref{eq_59},
we can see that
\begin{equation}\label{eq_61}
\mathcal H^+_-  \rightarrow \Phi, \qquad \mbox{when } \delta
\rightarrow 0.
\end{equation}
With the Mean Value Theorem, we also have
\begin{equation}\label{eq_62}
\VaR^{\delta} - 2c(i,a) + \VaR^{d}  \rightarrow 0, \qquad \mbox{when
} \delta \rightarrow 0, \ (i,a) \in \mathcal H^+_-.
\end{equation}
Substituting \eqref{eq_61}\&\eqref{eq_62} into \eqref{eq_60}, we
have
\begin{eqnarray}\label{eq_63}
\frac{\partial \Delta_{\CVaR}(\delta,d)}{\partial
\delta}\Big|_{\delta=0} \hspace{-0.3cm} &=& \hspace{-0.2cm}
\frac{1}{1-\alpha}\Bigg\{ \frac{\partial \VaR^{\delta}}{\partial
\delta} \hspace{-0.2cm}\sum_{i,a \in \mathcal H^+_-}
\pi^{\delta}(i,a)\Big|_{\delta=0} + \hspace{-0.2cm}\sum_{i,a \in
\mathcal H^+_-} \frac{\partial \pi^{\delta}(i,a)}{\partial \delta}
[\VaR^{\delta} - 2c(i,a) + \VaR^{d}] \Big|_{\delta
=0}\Bigg\} \nonumber\\
&=& \frac{1}{1-\alpha}\Big\{ 0 + 0 \Big\} = 0.
\end{eqnarray}
Therefore, substituting the above equation \eqref{eq_63} into
\eqref{eq_56}, we directly have
\begin{equation}
\frac{\partial \CVaR^{\delta}}{\partial \delta}\Big|_{\delta = 0} =
\sum_{i \in \mathcal S} \pi^{d}(i) \Bigg[\sum_{j \in \mathcal S}
[p(j|i,d'(i)) - p(j|i,d(i))] g^{d}(\VaR^d, j) +
\tilde{c}(\VaR^d,i,d'(i)) - \tilde{c}(\VaR^d,i,d(i)) \Bigg].
\nonumber
\end{equation}
Then Lemma~\ref{lemma4} is proved. \qed

\end{document}